\documentclass{amsart}

\usepackage{amsmath,amsxtra,amssymb,amsfonts,latexsym, mathrsfs, dsfont,bm,mathtools}
\usepackage{latexsym}
\usepackage{bbm}
\usepackage[usenames]{color}
\usepackage[hyperindex,pagebackref]{hyperref}

\usepackage[initials]{amsrefs}

\usepackage{comment}

\usepackage{graphics,epsf}         

\renewcommand{\theequation}{\arabic{section}.\arabic{equation}}
\newcommand\numberthis{\addtocounter{equation}{1}\tag{\theequation}}

\newtheorem{theorem}{Theorem}
\newtheorem{lemma}{Lemma}
\newtheorem{proposition}{Proposition}
\newtheorem{remark}{Remark}

\newtheorem{definition}{Definition}

\newtheorem*{theorem*}{Theorem}
\newtheorem*{lemma*}{Lemma}
\newtheorem*{proposition*}{Proposition}
\newtheorem*{remark*}{Remark}
\newtheorem*{example*}{Example}
\newtheorem*{exercise*}{Exercise}
\newtheorem*{definition*}{Definition}
\newtheorem*{corollary*}{Corollary}
\newtheorem*{notation*}{Notation}
\newtheorem*{claim*}{Claim}

\renewcommand{\q}{\quad}

\newcommand{\p}{\partial}

\newcommand{\wh}{\widehat}

\newcommand{\Id}{{\bf 1}}

\def\supp{{\text{\rm supp }}}

\begin{document}

\title[Uniform estimates]{Uniform estimates for Bilinear Hilbert Transform and Bilinear Maximal functions associated to polynomials}

\author{Xiaochun Li}
\address{Department of Mathematics, University of Illinois at Urbana-Champaign, Urbana, IL, 61801, USA}
\email{xcli@math.uiuc.edu}

\author{Lechao Xiao}
\address{Department of Mathematics, University of Illinois at Urbana-Champaign, Urbana, IL, 61801, USA}
\email{xiao14@math.uiuc.edu}

\date{\today}


\begin{abstract}
We study the bilinear Hilbert transform and bilinear maximal functions associated to polynomial curves and 
obtain uniform $L^r$ estimates for $r>\frac{d-1}{d}$ and this index is sharp up to the end point. 
\end{abstract}

\maketitle

\section{Introduction}
\label{intro}
Let $d\geq 2$ be a positive integer and $P(t) = a_dt^d+a_{d-1}t^{d-1}+\dots +a_2t^2$ be a polynomial of degree $d$ without linear term, 
where $a_d,\dots, a_2\in\mathbb{R}$. Let $\Gamma_P =(t,P(t))$, we define the bilinear Hilbert transform along $\Gamma_P$ as
\begin{equation} \label{defofH}
H_{\Gamma_P}(f,g)=p.v.\int_{\mathbb R} f(x-t)g(x-P(t))\frac{dt}{t}\,.
\end{equation}   
where $f(x)$ and $g(x)$ are Schwartz functions on $\mathbb{R}$.

For the bilinear Hilbert transform $H_{\Gamma}$,  
the $L^2\times L^2$ to $L^1$ estimate was studied in \cite{li2008} for the monomial curve.  Recently, 
the $L^2\times L^2 \rightarrow L^1$ result for $H_\Gamma$ was proved  to be true for the non-flat curve in \cite{lie}. In this paper, we
establish the $L^p\times L^q \rightarrow L^r$ estimates for the full range $p, q, r$, except for the end point case 
$r=\frac{d-1}{d}$.  In addition,  the operator norm is bounded uniformly because it only depends on the degree of the polynomial.
The method employed in this paper can be also used to obtain the uniform estimates for the bilinear maximal function $M_{\Gamma_p}$,
defined by
\begin{equation}\label{defofM}
M_{\Gamma_P}(f,g)=\sup_{\epsilon>0}\frac{1}{2\epsilon}\int_{-\epsilon}^{\epsilon} \left|f(x-t)g(x-P(t))\right| d{t}.
\end{equation}
More precisely, we have


\begin{theorem}\label{mainTH}
Let $P(t)$ be a polynomial without linear term. Then the operators $H_{\Gamma_P}$ defined as in (\ref{defofH}) can be extended to a bounded operator from $L^{p_1}\times L^{p_2}$ to $L^r$ for $r>\frac{d-1}{d}$, $p_1$, $p_2>1$ and $\frac{1}{r} = \frac{1}{p_1}+\frac{1}{p_2}$. In addition, the bound is uniform in a sense that it depends on $d$ but  independent of the coefficients of $P$.  
\end{theorem}

\begin{theorem}\label{maxbi}
Let $p_1,p_2$ and $r$ be the same as Theorem \ref{mainTH}, then the bilinear maximal operator 
 $M_{\Gamma_P}$ can be extended to a bounded operator from $L^{p_1}\times L^{p_2}$ to $L^r$ and the bound is also uniform.   
\end{theorem}

This theorem on maximal functions can be viewed as a bilinear analogue of Theorem 3 in \cite{wr}.   It is not difficult to 
construct a polynomial $P$ with no linear term 
such that $H_{\Gamma_P}$ and $ M_{\Gamma_P}$ are unbounded whenever $r<{(d-1)}/d$
(see Section \ref{bad case}).  For general polynomials (containing linear term), the problem is harder, especially for the full range 
$p_1, p_2, r$, since it involves the results of the uniform estimates for the bilinear Hilbert transform (see \cite{gr}, \cite{li06}, \cite{th})
 and  for the bilinear  maximal  transform (see \cite{la}).   
The constraint of $r>(d-1)/d$ is superfluous if there is some additional convexity condition for the curve. 
For instance,  for the monomial curve, i.e. $\Gamma_d = (t, t^d)$ and $d\geq 2$, we can extend the range of $r$ to $(\frac{1}{2}, \infty)$, 

\begin{theorem}\label{seTH} 
Let $\frac{1}{r} = \frac{1}{p_1}+ \frac{1}{p_2}$ where $r>\frac{1}{2}$, $p_1$, $p_2>1$ and $\Gamma(t) = (t, t^d)$. Then $H_\Gamma$
($M_\Gamma$)  can be extended to a bounded operator from $L^{p_1}\times L^{p_2}$ to $L^r$. 
\end{theorem}
It is then natural to ask the following question: for a given polynomial $P(t)$, what is the lower bound of $r$ such that
 $H_{\Gamma_P}$ and $M_{\Gamma_P}$ map from $L^{p_1}\times L^{p_2}$ to $L^r$? The following theorem provides a complete answer to this question for $P(t)$ without linear term and characterizes the lower bound of $r$ by the decay factor of the level set estimate: $|\{t:|P'(t)-1|<h\}|$. 
\begin{theorem}\label{P&r} 
Let $P(t)$ be a polynomial without linear term, the following are equivalent:\\
{\rm(i)} All the roots of $P'(t)-1=0$ has order at most (k-1);\\
{\rm(ii)} There is a constant $C_P=C(P)$ s.t. the following level set estimate is true for $h$ sufficiently small
 $$|\{t:|P'(t)-1|<h\}|< C_Ph^{\frac{1}{k-1}}.$$
{\rm(iii)}  $H_{\Gamma_P}$ and $M_{\Gamma_P}$ maps from $L^{p_1}\times L^{p_2}$ to $L^r$ for all $r>\frac{k-1}{k}$, $p_1$, $p_2>1$ satisfying $\frac{1}{r} = \frac{1}{p_1}+ \frac{1}{p_2}$.
\end{theorem}
Although we don't have a uniform boundedness as in Theorem \ref{mainTH} and \ref{maxbi}, still a weaker one holds: 
the bounds of the operators $H_{\Gamma_P}$ and $M_{\Gamma_P}$ depend only on the degree $d$ and the constant $C_P$ in (ii) above. 
The whole picture of this point, as well as Theorem \ref{P&r}, will become transparent in Section \ref{bad case}.

This type of problems is not only motivated by the classical questions on Hilbert transforms along curves (see \cite{S-W}), but also 
arises from the non-conventional ergodic average (see \cite{Furst}, \cite{li2008}). For instance, for $n\in\mathbb N$, consider 
\begin{equation}
 T^*(f_1, f_2)(n) = \sup_{M\in\mathbb N} \big |\frac{1}{M}  \sum_{m=1}^M f_1(n-m) f_2(n-m^2)\big|\,.
\end{equation}
In \cite{Hu-Li}, it was shown that $T^*$ can be extended to a bounded operator from $\ell^2\times \ell^2$ to $\ell^r$ provided that $r>1$. 
Because there is no transference principle available, it is not clear that Theorem \ref{maxbi} implies any boundedness of $T^*$.
A very interesting question is to build up $\ell^2\times \ell^2\rightarrow \ell^1$ estimate for $T^*$, from which 
the pointwise convergence of the non-conventional dynamic system follows. The circle method and(or) large sieve method 
are expected to resolve the problem.

The method used in this paper essentially works for more general curves on nilpotent groups. We shall not pursue this 
in this article. Besides the generalisation to more general curves, it is natural to ask whether one can extend the results to 
the multilinear cases and/or the higher dimensional cases (see \cite{li2008}).


\section{Main Structure of the Proof}
\setcounter{equation}0

Unfortunately, the proof of Theorem \ref{mainTH} has to be quite technical, since it involves
the uniform bounds. 
In this section, we sketch the proof of Theorem \ref{mainTH} to present a clear picture for the reader. 
In order to build up the uniform estimates, 
first we introduce some technical notations associated to the polynomial. 
Let $P(t) = \sum\limits_{k=1}^{d}a_kt^k$ and $N$ be a sufficiently large positive integer,  say,  $N>2^{100d!}$. 
For $l =1, \dots, d$, we define $J_l(N)$ as 
\begin{equation}\label{good}
J_l(N)= \left\{j\in\mathbb{Z}: {|j|\geq N},\, |a_l(2^{-j})^l|>2^{N+2d}|a_k(2^{-j})^k|, \,\mbox{for all\,}k\neq l  \right\}.
\end{equation}
Then set $J_{\rm bad}(N)$ and $J_{\rm good}(N)$, respectively,  to be 
\begin{equation}
J_{\rm bad}(N) := \bigcup_{l=1}^{d}J_l(N),
\end{equation}
and $$ J_{\rm good}(N) := \mathbb{Z}\backslash J_{\rm bad}(N).$$
 The $J_l(N)$ may be empty for some $l$ and can be considered essentially as the collection of 
 dyadic numbers at which the $l$-th term in the polynomial  $P$ dominates all other terms. 
Henceforth,  whenever $j\in J_l(N)$ and $|t| \sim 2^{-j}$, the polynomial $P$ behaves almost the same as the monomial $a_lt^l$. 
The following lemma asserts that the cardinality of $J_{\rm good}(N)$ is 
majorized by a constant which is independent of the coefficients of $P(t)$. 
This uniform upper bound is crucial in our proof for  Theorem \ref{mainTH}.
\begin{lemma}\label{keyL}The upper bound of the cardinality of $J_{\rm good}(N)$ depends only on $N$ and the $degree$ d of $P(t)$, more precisely  
$$\#J_{\rm good}(N)\leq (2(N+2d)+1)d(d-1)+(2N-1).$$ 
\end{lemma}
\begin{proof}
The proof is a simple application of pigeonhole principle. Let $|a_l|= 2^{b_l}  $. Observe that if $j\in J_{\rm good}(N)$, then there exists two integers $1\leq m\neq n\leq d $ such that 
\begin{align}\label{mn}
m(-j)+b_m+(N+2d)\geq n(-j)+b_n\geq m(-j)+b_m-(N+2d).
\end{align}
Thus
$$J_{\rm good}(N) \subset \left(\bigcup_{1\leq m\neq n \leq d}J_{\rm good}(N, m, n)\right)\bigcup \{|j|<N\},$$ 
where 
\begin{align*}
J_{\rm good}(N, m, n) = \left\{j\in J_{\rm good}(N): (m,n) \mbox{ satisfying  } (\ref{mn})\right\}.
\end{align*}
The cardinality of $\{|j|<N\}$ equals $(2N-1)$ and by (\ref{mn}), the cardinality of $J_{\rm good}(N, m, n)$ is at most $(2(N+2d)+1)$ .
Noticing that the number of different pairs $(m,n)$ is $d(d-1)$, the conclusion of the lemma is obvious now. 
\end{proof}
\begin{remark}
The reader may notice that the upper bound in the above lemma is far from sharp, however it is enough for our application. 
\end{remark} 

The quantity $N$ is chosen to depend only on $d, p_1$ and $p_2$,  but sufficiently large for the technical reason. 
Basing on Lemma \ref{keyL}, we will decompose the operator $H_{\Gamma_P}(f,g)$ into two components.
 First, let $\frac{1}{t} = \sum\limits_{j\in \mathbb{Z}}\rho_j(t)$, where $\rho_j(t)$ is a smooth odd function, supported on 
$(2^{-j-1},  2^{-j+1})\cup (-2^{-j+1},  -2^{-j-1})$. Indeed, $\rho_j(t) =2^j\rho(2^jt)$, where $\rho(t) = \rho_0(t)$. Set $T_j(f,g)(x)$ as
\begin{equation}\label{defofTj}
T_j(f,g)(x) = \int f(x-t)g(x-P(t))\rho_j(t) dt. 
\end{equation}
We then write the bilinear Hilbert transform $H_{\Gamma_P}$ as
\begin{equation}\label{two-comp}
 H_{\Gamma_P} = \sum_{j\in J_{good}(N)} T_j  + \sum_{j\in J_{bad}(N)} T_j\,. 
\end{equation}
Lemma \ref{keyL} tells us that $\#J_{\rm good}(N)$ is dominated by a constant only depending on $d, p_1, p_2$ and hence, 
to estimate the first term in  (\ref{two-comp}),  it suffices to control $T_j$ for $j\in J_{\rm good}(N)$ individually. The following theorem achieves this goal.

\begin{theorem}\label{badL}
Let $P(t)$ be a polynomial of degree $d$ with no linear term. Then for $ r>\frac{d-1}{d}$, $p_1>1$, $p_2>1$ with $\frac{1}{p_1}+\frac{1}{p_2} = \frac{1}{r}$, there is a constant $C$ independent of $j$ and the coefficients of $P$, such that
$$\|T_j(f,g)\|_r\leq C\|f\|_{p_1}\|g\|_{p_2}.$$ 
\end{theorem}
The case when $r\geq 1$ is trivial and a finer analysis is required for $r<1$. 
The tool we use here is the  van der Corput type estimate for obtaining the uniform bound. 
 We postpone the proof of Theorem \ref{badL}
to Section \ref{bad case},  without sidetracking the reader from the main structure of the proof. 
We now turn to  the case $j\in J_{\rm bad}(N)$ for getting an estimation of the second term in \ref{two-comp}.  To fulfill this, we need the following
theorem.
\begin{theorem}\label{goodL} 
Let $P(t)$ be a polynomial of degree $d$ without linear term and $r>\frac{1}{2}$, $p_1, p_2>1$ with $\frac{1}{p_1}+\frac{1}{p_2}=\frac{1}{r}$. If $N$ is sufficiently large, then for each $2\leq l\leq d$, there is a constant $C$  independent of the coefficients of $P(t)$ such that   
$$
\Big\|\sum\limits_{j\in J_l(N)}T_j(f,g)\Big\|_r\leq C\|f\|_{p_1}\|g\|_{p_2}.
$$ 
\end{theorem}   

This theorem requires a delicate time-frequency analysis. Sections \ref{decomposition3}, \ref{section5}, \ref{rnot1}
will be devoted to the proof of Theorem \ref{goodL}. 
The idea in the proof can be outlined as follows. Observe that if $j\in J_l(N)$, then for $2^{-j-1}<|t|< 2^{-j+1}$,  $|P(t)-a_lt^l| \leq 2^{-N}|a_lt^l|$. Because we can choose $N$ sufficiently large, the polynomial $P$ behaves like the monomial $a_lt^l$, except for a tiny perturbation $|P(t)-a_lt^l|$. Although technically some treatment is required for taking care of the perturbation, the polynomial  $P$ can be 
viewed as a monomial locally. The main ingredient for handling the perturbation is a theorem on inverse functions, which asserts
that a small perturbation of a sufficiently nice function can not produce a significant change for its inverse function.
The details of it will appear in Appendix.    
We adapt the uniformity method in \cite{li2008} to obtain a uniform bound from $L^2\times L^2$ to $L^1$ with a decay factor, 
after removing paraproducts (see Proposition \ref{L1}). 
To make $r$ go below 1, the method of time frequency analysis has to be invoked (see Proposition \ref{below1}).  
For $r<1$, the main issue here is to obtain an appropriate upper bound that grows slowly enough, in contrast to  the 
decay factor mentioned above.  This is achieved by a trick related to Whitney decomposition in Section \ref{P_error}
and the time-frequency analysis. Finally, Theorem \ref{goodL} follows from interpolation. 
Combining Theorem \ref{badL} and Theorem \ref{goodL} we then obtain Theorem \ref{mainTH}. \\


Theorem \ref{seTH} is  a direct consequence of Theorem \ref{goodL}. 
In fact,  when $P(t) = t^d$ every integer $j$ is either in $J_{\rm bad}(N)$
or in the interval $(-N, N)$.  For $j\in J_{\rm bad}(N)$, Theorem {\ref{goodL}} implies the desired result. 
For $|j|<N$, the theorem follows from the single scale estimate,  presented in Section \ref{bad case}. 
\\

The reader may also notice that Theorem \ref{badL} is the only constraint preventing $r$ go below $\frac{d-1}{d}$. If not pursuing the uniform boundedness,  we are able to obtain Theorem \ref{P&r} by modifying the proof of Theorem \ref{badL} and combining the result of Theorem \ref{goodL}. Details can be found in section \ref{bad case}.

In what follows, we use $C$ to denote a constant independent of the coefficients of $P$. The exact value of $C$ may vary from line to line, but it is unimportant.


\section{Treatments for the case of $J_{\rm good}(N)$ }\label{bad case}
\setcounter{equation}0

The first part of this section is devoted to prove Theorem \ref{badL}. 
In the end, we will give a counterexample to show that $r\geq\frac{d-1}{d}$ is necessary,  
and provide a proof of Theorem \ref{P&r}. The main tool in proving Theorem \ref{badL} is the van der Corput type estimate for the level set of a polynomial (or a function). The following lemma is well-known and can be found, for example, in \cite{st} and \cite{ca}.
\begin{lemma}\label{van}
For each $k\geq 1$, there exists an absolute constant $C_k$ such that for any function $u(t)$ satisfying $u^{(k)}(t)\geq 1$ for all $t$, then
$$
|\{t:|u(t)|\leq \alpha\} |\leq C_k \alpha^{1/k}.
$$ 
\end{lemma}
\subsection{Proof of Theorem \ref{badL}}\label{proof_badL}
First, by insertting the absolute value, we have
\begin{align*}
\|T_j(f,g)(x)\|_1&\leq \int \int |f(x-t)||g(x-P(t))|dx\, |\rho_j(t)|dt
\\
&\leq \|f\|_p\|g\|_{p'} \|\rho\|_1, 
\end{align*}
which is bounded by $C\|f\|_p\|g\|_{p'}$, for all $1\leq p\leq \infty$ and $p'$ is the conjugate of $p$. Interpolating with the trivial 
bound $L^{\infty}\times L^{\infty}\to L^{\infty}$, we obtain  Theorem \ref{badL} immediately  for $r\geq 1$.  
We now turn to the difficult case when  $r<1$. Observe that  
$$
T_j(f,g)(x) = \int f(x-2^{-j}t)g(x-P(2^{-j}t))\rho(t) dt.
$$
 It is then natural to compare the measures of the following two sets
\begin{align}\label{A_10}
A_1= \{2^{-j}t: t\in \mbox{supp} \, \rho(t)\}
\end{align}
and
\begin{align}\label{A_P0}
A_P= \{P(2^{-j}t): t\in \mbox{supp} \, \rho(t)\}.
\end{align}
Let 
\begin{align*}
Q_j(t) = (P(2^{-j}t))'-(2^{-j}t)' = da_d(2^{-j})^dt^{d-1}+\cdots +2a_2 (2^{-j})^2t-2^{-j}
\end{align*}
and 
$$
E_\alpha = \left\{ t\in \mbox{supp}\, \rho: \alpha \leq |(P(2^{-j}t))'| \leq 2\alpha  \right\},
$$
where $\alpha$ is a dyadic number. Define
\begin{align}\label{A_1}
A_1(\alpha)= \{2^{-j}t: t\in E_\alpha\}
\end{align}
and
\begin{align}\label{A_P}
A_P(\alpha) = \{P(2^{-j}t): t\in E_\alpha\}.
\end{align}
$T_j$ can then be decomposed into $T_j(f,g)(x) = \sum_{\alpha} T_j^\alpha (f,g)(x)$, where 
\begin{equation}\label{defofTj-al}
T_j^\alpha (f,g)(x) = \int_{E_\alpha} f(x-2^{-j}t)g(x- P(2^{-j}t)) \rho(t)dt.
\end{equation}

We need to establish the $L^r$ estimate for the sum $\sum_\alpha T_j^\alpha$. 
To achieve it, divide the range of dyadic number $\alpha$ into 3 cases:
\\
$\bullet$ {Case 1}. \q$\alpha \geq 2\cdot 2^{-j}$;
\\
$\bullet$ {Case 2}. \q$\alpha \leq \frac{1}{4}\cdot2^{-j}$;
\\
$\bullet$ {Case 3}.  $ \q\frac{1}{2}\cdot 2^{-j}\leq \alpha \leq  2^{-j}$. \\
Case 1 and Case 2 are similar, while Case 3 is the hardest. We begin with Case 1. Define constants $b_k$'s for $2\leq k\leq d$ as follows:
\begin{align}\label{bk}
b_k = |a_k| \cdot (2^{-j})^k \cdot2^{3k} \cdot k! 
\end{align}
and assume
$$
b_{k_0} = \max_{2\leq k\leq d}\{b_k\}. 
$$
If $\alpha \geq 2\cdot 2^{-j}$, then 
\begin{align}
2\alpha |E_\alpha| \geq |A_P(\alpha)|\geq \alpha |E_\alpha|>2^{-j}|E_\alpha| =  |A_1(\alpha)|.
\end{align}
Hence, in estimating the $L^r$ norm of $T^\alpha_j(f,g)(x)$, 
we are able to assume $x$ is supported on an interval $I_\alpha$ of length $2\alpha |E_\alpha|$. Applying change of variables $u = x-2^{-j}t$, $v = x-P(2^{-j}t)$, we obtain 
\begin{align*}
&\q\int_{I_\alpha}\left|\int_{E_\alpha} f(x-2^{-j}t)g(x-P(2^{-j}t))\rho(t)dt\right|^{1/2}dx
\\
&\leq C|I_{\alpha}|^{1/2} \left(\int_{I_\alpha}\int_{E_\alpha}| f(x-2^{-j}t)g(x-P(2^{-j}t))|dtdx\right)^{1/2}
\\
& \leq C|I_{\alpha}|^{1/2} \left(\int \int |f| |g | \bigg|\frac{\partial (x,t)}{\partial(u,v)}\bigg|\Id_{E_\alpha}(t)dudv \right)^{1/2}
\\
&
\leq C|I_{\alpha}|^{1/2} (\frac{1}{2}\alpha)^{-1/2}\left( \|f\|_1\|g\|_1 \right)^{1/2} 
\,\, \leq C|E_{\alpha}|^{1/2} ( \|f\|_1\|g\|_1 )^{1/2} 
\end{align*}
For the last two inequalities, we used the facts
\begin{align*}
\bigg|\frac{\partial(u,v)}{\partial(x,t)} \bigg| = |(P(2^{-j}t))' - (2^{-j}t)'| \geq \frac{1}{2}\alpha\q\q \mbox{for all} \,\, \,t\in E_\alpha,
\end{align*}
and 
\begin{align*}
|I_\alpha| \sim \alpha |E_\alpha|. 
\end{align*}
Henceforth, when $\alpha\geq 2\cdot 2^{-j}$ we have
\begin{equation} 
\|T_j^\alpha(f,g)\|_{1/2}^{1/2} \leq C|E_\alpha|^{1/2}(\|f\|_1\|g\|_1)^{1/2}.
\end{equation}

Now we consider the second case: $\alpha \leq \frac{1}{4}2^{-j}$. In this case, observe that 
$$
2\cdot 2^{-j}|E_\alpha| \geq |A_1(\alpha)| >|A_P(\alpha)|,
$$
hence when considering $\|T_j^\alpha(f,g)\|_r^{r}$ we may restrict $x$ on an interval $I_\alpha$ of length $2\cdot2^{-j}|E_\alpha|$. Again, we 
perform change of variables $u= x-2^{-j}t$, $v= x-P(2^{-j}t)$ and notice that $|\frac{\partial(u,v)}{\partial(x,t)}| \geq \frac{1}{2}2^{-j}$, 
as we did in the case 1, we obtain
\begin{equation}
\int_{I_\alpha}\left|\int_{E_\alpha}f(x-2^{-j}t)g(x-P(2^{-j}t))\rho(t)dt\right|^{1/2}dx
\leq C|E_\alpha|^{1/2}(\|f\|_1\|g\|_1)^{1/2}.
\end{equation}

The next step is to estimate the level set $E_\alpha$. First notice that
\begin{align}\label{levelset}
|(P(2^{-j}t))^{(k_0)}| \geq \frac{1}{2} |a_{k_0} (2^{-j})^{k_0}k_0!|,
\end{align}
which follows from the definition of $b_k$ and $b_{k_0}$ in (\ref{bk}). Then by Lemma \ref{van}, 
\begin{align*}
|E_\alpha| \leq |\{   |(P(2^{-j}t))'|  \leq 2\alpha\}| \leq C_d \left(\frac{2\alpha}{\frac{1}{2}|a_{k_0} (2^{-j})^{k_0}k_0!|} \right)^{\frac{1}{k_0-1}}.
\end{align*}
In order to prove that $\sum_\alpha |E_\alpha|^{1/2}$ is bounded above uniformly, we just need to show $\frac{2\alpha}{\frac{1}{2}|a_{k_0} (2^{-j})^{k_0}k_0!|}$ bounded above uniformly. Indeed, by the definition of $b_{k_0}$ we have
\begin{align*}
\alpha &\leq \sup\limits_{t\in supp \, \rho}|(P(2^{-j}t))'|
\\
&\leq |da_d(2^{-j})^dt^{d-1})|+\cdots +|2a_2(2^{-j})^2t|
\leq d|b_{k_0}|
= d \cdot k_0!\cdot2^{3k_0}\cdot |a_{k_0}(2^{-j})^{k_0}|.
\end{align*}
This implies $\frac{\alpha}{|a_{k_0}(2^{-j})^{k_0}k_0!|} \leq d\cdot 2^{3k_0}$ and hence $\sum_\alpha|E_\alpha|^{1/2} \leq C_d$. Summing up all $\alpha$ in Case 1 and Case 2, we obtain
\begin{align}\label{goal12}
\Big\|\sum\limits_{\alpha\leq \frac{1}{4}2^{-j}\,\mbox {or}\,\alpha \geq 2\cdot 2^{-j}}T_j^\alpha(f,g)\Big\|_{1/2} \leq C_d\|f\|_1\|g\|_1
\end{align}
and
\begin{align}\label{goal123}
\Big\|\sum\limits_{\alpha\leq \frac{1}{4}2^{-j}\,\mbox {or}\,\alpha \geq 2\cdot 2^{-j}}T_j^\alpha(f,g)\Big\|_{r} \leq C_d\|f\|_{p_1}\|g\|_{p_2}
\end{align}
for $\frac{1}{2}\leq r\leq 1$ by interpolation.
We turn to the last case, where we need to consider 
$ \{t\in \mbox{supp} \,\rho: \frac{1}{2}\cdot 2^{-j}< |(P(2^{-j}t))'| < 2\cdot 2^{-j} \}$ in a finer scale. We only need to focus on the case $(P(2^{-j}t))'>0$, because the case when $(P(2^{-j}t))'<0$ can be handled exactly  same as Case 2. Let

$$
E_0 =\Big\{t\in \mbox{supp} \,\rho: \frac{1}{2}\cdot 2^{-j}< (P(2^{-j}t))' < 2\cdot 2^{-j} \Big\}
$$ 
and
$$
E_0(h) =\Big\{t\in E_0 \,: h\cdot 2^{-j}\leq  |(P(2^{-j}t))' - (2^{-j}t)'|\leq 2h\cdot 2^{-j} \Big\},
$$ 
where $0<h\leq 1$ is a dyadic number and $E_0 =\bigcup_{0<h\leq 1} E_0(h)$ . Let
$$
T_{j,h}(f,g)(x) = \int_{E_0(h)} f(x-2^{-j}t)g(x-P(2^{-j}t))\rho(t)dt.
$$
Our goal is to prove that there is a positive number $\epsilon>0$ such that the following inequality is true for $0<h\leq 1$, $\frac{d-1}{d}<r<1$, $p_1$, $p_2>1$ with $\frac{1}{p_1}+\frac{1}{p_2}= \frac{1}{r}$
\begin{align}\label{hgoal}
\|T_{j,h}(f,g)\|_r \leq C h^\epsilon \|f\|_{p_1}\|g\|_{p_2}.
\end{align}
We begin with the estimation of $E_0(h)$. Notice that $k_0\geq 2$, the following inequality is exactly (\ref{levelset}) 
$$
\Big|\Big((P(2^{-j}t)-2^{-j}t)'\Big)^{(k_0-1)}\Big| \geq \frac{1}{2}\cdot |a_{k_0}|(2^{-j})^{k_0}\cdot k_0!.
$$ 
In addition, for $t\in E_0$, one has
\begin{align*}
\frac{1}{2}\cdot 2^{-j} \leq \sup\limits_{t\in E_0}|(P(2^{-j}t))'|\leq d\cdot b_{k_0}.
\end{align*}
Hence
\begin{align}\label{op}
\frac{1}{2d}\cdot 2^{-j} \leq b_{k_0} = k_0!\cdot2^{3k_0}\cdot |a_{k_0}(2^{-j})^{k_0}|.
\end{align}
Applying Lemma \ref{van} and (\ref{op}), we obtain
\begin{align}\label{levelseth}
|E_0(h)| \leq \left|\frac{2h\cdot 2^{-j}}{\frac{1}{2}|a_{k_0}(2^{-j})^{k_0}k_0!|}\right|^{\frac{1}{k_0-1}}\leq C_d h^\frac{1}{k_0-1}.
\end{align}
Because $|\{P(2^{-j}t): t\in E_0(h)\}| \sim |\{2^{-j}t: t\in E_0(h)\}| \leq 10 \cdot2^{-j}|E_0(h)|$,  in proving (\ref{hgoal}) 
we can assume $x$ is supported on an interval $I_h$ of length $10\cdot 2^{-j}|E_0(h)|$. Thus
\begin{align*}
&\int_{I_h} \int_{E_0(h)}\left|f(x-2^{-j}t)g(x-P(2^{-j}t))\right|dtdx
\\
\leq & \int_{I_h} \left(\int_{E_0(h)}|f(x-2^{-j}t)|^{p}dt\right)^{1/p}\left(\int_{E_0(h)}|g(x-P(2^{-j}t))|^{p'}dt\right)^{1/p'}dx
\\
\leq & \int_{I_h} (2^j)^{1/p}\|f\|_p(2^{j})^{1/p'}\|g\|^{p'}dx \leq 10\cdot |E_0(h)|\|f\|_p\|g\|_{p'},
\end{align*}
that is 
\begin{align}\label{goal1}
\|T_{j,h}(f,g)\|_1 \leq C |E_0(h)|\|f\|_p\|g\|_{p'}.
\end{align}
On the other hand, applying Cauchy-Schwarz inequality, we have  
\begin{align*}
  & \q \left(\int_{I_h}\left(\int_{E_0(h)} |f(x-2^{-j}t)g(x-P(2^{-j}t))|dt\right)^{1/2}dx\right)^2  
\\
&\leq |I_h| \int_{I_h}\int_{E_0(h)} \left|f(x-2^{-j})g(x-P(2^{-j}t))\right|dtdx
\\
&\leq|I_h| \int \int |f(u)||g(v)|\left|\frac{\partial(x,t)}{\partial(u,v)}\right|\Id_{E_0(h)}(t)dudv
\\
& \leq C |E_0(h)|h^{-1}\|f\|_1\|g\|_1.
\end{align*}
Here we applied change of variables $u = x- 2^{-j}t$, $v = x - P(2^{-j}t)$ and the fact that for $t\in E_0(h)$
$$
\Big|\frac{\partial(x,t)}{\partial(u,v)}\Big| \leq C 2^j\cdot h^{-1}.
$$
Thus we obtain 
\begin{align}\label{goal2}
\|T_{j,h}(f,g)\|_{1/2}\leq C |E_0(h)|h^{-1}\| f \|_1\| g \|_1.
\end{align}
Combining the results of (\ref{levelseth}), (\ref{goal1}), (\ref{goal2}) and applying interpolation, we obtain for $r>\frac{k_0-1}{k_0}$
\begin{align*}
\|T_{j,h}(f,g)\|_{r} \leq C h^{\epsilon}\|f\|_{p_1}\|g\|_{p_2},
\end{align*}
for some $\epsilon>0$ and thus we have 
\begin{align}\label{goal0}
\|\sum\limits_{0<h\leq 1}T_{j,h}(f,g)\|_{r} \leq C\|f\|_{p_1}\|g\|_{p_2}.
\end{align}
(\ref{goal123}) and (\ref{goal0}) together complete the proof of the inequality
$$
\|T_j(f,g)\|_{r} \leq C\|f\|_{p_1}\|g\|_{p_2} \q \mbox{for} \q \frac{d-1}{d}<r \leq 1.
$$
This finishes the proof of Theorem \ref{badL}.

\subsection{A counterexample}\label{counter}
We show that $r\geq \frac{d-1}{d}$ is necessary for the boundedness of $H_{\Gamma_P}$, 
even for a single scale. \\
Let $P(t) = t+(\frac{1-t}{A})^{d}-\frac{1}{A^d}$ and take $ A = d^{1/d}$ to eliminate the linear term. 
Let $f =\Id_{[0,\delta]} $ and $g=\Id_{[\frac{1}{A^d},\frac{1}{A^d}+\delta]}$ be characteristic functions  and $\delta$ is a small positive number.
 Set $B =\frac{A}{10}$ and consider $T_0(f,g)(x) =\int f(x-t)g(x-P(t))\rho(t)dt$, for $x\in[1+B\delta^{1/d},1+2B\delta^{1/d}]$. If $t<0$, $(x-t)\geq 1>\delta$ and hence $f(x-t)g(x-P(t))\rho(t)=0$. Thus we only need to focus on $t>0$ and $\rho(t) >0$. In this case, $f(x-t)g(x-P(t))\rho(t)\geq 0$, for all $x\in[1+B\delta^{1/d},1+2B\delta^{1/d}]$.  Now for fixed $x$, when $t\in[x-\delta/2,x-\delta/4]$ we have
\begin{align*} 
\frac{1}{A^d}<x-P(t)
<\delta+\frac{1}{A^d}.
\end{align*}
That is, for $x\in[1+B\delta^{1/d},1+2B\delta^{1/d}]$ and $t\in[x-\delta/2,x-\delta/4]$ we have
$$
f(x-t)g(x-P(t))= 1
$$
and 
$$
f(x-t)g(x-P(t))\rho(t)\geq 1/2.
$$
Hence, for every $x\in[1+B\delta^{1/d},1+2B\delta^{1/d}]$ we have
$$
T_0(f,g)(x) \geq 1/2 \cdot \delta/4 = \delta/8.
$$
This implies
$$
\|T_0(f,g)(x)\|_r \geq (B\delta^{1/d})^{1/r} \delta/8.
$$
Observe that $\|f\|_{p_1} = \delta^{1/p_1} $ and $\|g\|_{p_2} = \delta^{1/p_2}$. Then
$\|T_0(f,g)\|_r \leq C\|f\|_{p_1}\|g\|_{p_2}$ yields
$$
(B\delta^{1/d})^{1/r} \delta/8 \leq C\delta^{1/p_1+1/p_2} = C\delta^{1/r}.
$$
Let $\delta\to 0^+$, then we obtain $r\geq \frac{d-1}{d}$ as desired.

\subsection{Proof of Theorem \ref{P&r}}
Since $P(t)$ is a polynomial of bounded degree, the equivalence between (i) and (ii) in Theorem \ref{P&r} is obvious. To show (ii) implies (iii), we proceed as in subsection \ref{proof_badL} and use the same notation. The arguments for Case 1: $\alpha>2\cdot 2^{-j}$ and Case 2: $\alpha<\frac{1}{4} \cdot 2^{-j}$ are exactly the same and hence we still have (\ref{goal12}) and ({\ref{goal123}). For case 3, we also have (\ref{goal1}) and (\ref{goal2}) true. The only difference is that the level set estimate (\ref{levelset}) is replaced by (ii) in Theorem \ref{P&r}
\begin{align}\label{new_level}
|E_0(h)| \leq Ch^{\frac{1}{k-1}}.
\end{align}
By interpolation, we have for $r>\frac{k-1}{k}$
\begin{align}\label{k} 
\|T_{j,h}(f,g)\|_{r} \leq C h^{\epsilon}|\|f\|_{p_1}\|g\|_{p_2},
\end{align}
and thus we get the same result as in Theorem \ref{badL} for $r>\frac{k-1}{k}$. Combining Theorem \ref{goodL} we obtain Theorem \ref{P&r} (iii). 
\\
Now we turn to prove (iii) implies (i). The proof is similar to what we did in Subsection \ref{counter}. 
Assume $t_0$ is a root of $P'(t)-1=0$ of order $k_0$. Since $P(t)$ has no linear term (and no constant term), then $t_0\neq 0$. Let $0<\delta<<|t_0|$ be a small number and define $f = \Id_{[-\delta,\delta]}$ and $g=\Id_{[t_0-P(t_0)-\delta, t_0-P(t_0)+\delta]}$. Fix $x\in[t_0-\delta^\frac{1}{k_0+1}/A, t_0+\delta^\frac{1}{k_0+1}/A]$, where $A$ is a large constant depending on $P$ and to be specific later. Restrict $t$ to $|t-x|<\delta/100$. Notice that $t_0$ is a root of $\big((t-P(t))-(t_0-P(t_0))\big)$ of order $(k_0+1)$. Then when $A$ large enough, we have 
$$
|(t-P(t))-(t_0-P(t_0)|\leq \delta/100.
$$
Thus, we obtain 
$$
|(x-P(t))- (t_0-P(t_0))| = |(x-t) +\big((t-P(t))-(t_0-P(t_0))\big)| <\delta/10
$$
and  
$$
f(x-t)g(x-P(t))=1
$$
for $x\in[t_0-\delta^\frac{1}{k_0+1}/A, t_0+\delta^\frac{1}{k_0+1}/A]$ and $|t-x|<\delta/100$. This yields
$$
\|H_{\Gamma_P}(f,g)\| \geq C \delta^{\frac{1}{r(k_0+1)}}\delta
$$
Combining $\|f\|_{p_1}\|g\|_{p_2} = C \delta^{\frac{1}{r}}$ and Theorem (iii), we obtain $k_0\leq k-1$. 

\section{ Decomposition of the operator $T_j$}\label{decomposition3}
\setcounter{equation}0

The purpose of the section is to decompose the operator $T_j(f,g)$ into two parts in terms of pseudo-differential operators, as in \cite{li2008}. 
One part can be handled by the uniform estimates of some paraproducts (see \cite{li2008U}), while the other part requires most work.
Set $j_l =\frac{b_l}{l-1}$ where $|a_l| = 2^{b_l}$.  The technical issue here is that  $j_l$ may not be an integer. 
To overcome this difficulty, we simply shift $J_l(N)$ by $j_l$ units as follows so that $j_l$ can be essentially treated as an integer in the proof. 
More precisely,  first, by the proof of Lemma \ref{keyL}, $J_l(N)$ is ``continuous" in the sense that $j \in J_l(N)$
 whenever $j$ satisfies $\inf J_l(N) \leq  j \leq \sup J_l(N)$. Thus we set
$$
J^*_l(N) = \{j\in\mathbb Z : \inf J_l(N) \leq j+ j_l \leq \sup J_l(N) \}
$$
and
$$
{\mathcal E} = \Big\{\bigcup_{j\in J_l(N)} [2^{-j-1},2^{-j+1}] \Big\} \backslash \left\{\bigcup_{j\in J_l^*(N)}[2^{-(j+j_l)-1}, 2^{-(j+j_l)+1}] \right\}.
$$
The set $\mathcal E$ is empty whenever $j_l$ is an integer, and can be covered by finitely many dyadic intervals. The number of such intervals 
is dominated by a uniform constant. Henceforth,    
 the case $|t|\in \mathcal E $ can be incorporated into the case of $J_{\rm good}(N)$. 
Therefore, we only need to focus on $|t|\in\cup_{j\in J_l^*(N)}[2^{-(j+j_l)-1}, 2^{-(j+j_l)+1}]$, 
and decompose the operator $T_{j+j_l}$ for $j\in J_l^*(N)$. Here $j_l$ can be viewed as an integer due to the previous shifting argument.

Recalling the definition of $T_j$ as in (\ref{defofTj}) and applying Fourier Inversion formula to both $f(x)$ and $g(x)$, we get 
 $$
 T_{j+j_l}(f,g)(x) =  \int \int \hat f(\xi) \hat g(\eta) e^{2\pi i (\xi+\eta)x} 
\Big( \int e^{-2\pi i \big(t\xi +(a_lt^l+P_l(t))\eta\big)}\rho_{j+j_l}(t)dt \Big)
d\xi d\eta
 $$
 Here $P_l(t)=P(t)-a_l t^l$. 
The next step is to decompose the support of $\xi$ and $\eta$ dyadically. Let $\Theta$ be a Schwartz function supported on $(-1,1)$ such that $\Theta(\xi)=1 $ if $|\xi|\leq \frac{1}{2}$. Let $\Phi(\xi)$ be a Schwartz function such that 
$$\hat\Phi(\xi) = \Theta(\frac{\xi}{2})-\Theta(\xi).$$
Then $\Phi$ is Schwartz function such that $\hat\Phi$ is supported on $\{\xi: \frac{1}{2}<|\xi|< 2\}$ and 
$$
\sum\limits_{m\in\mathbb{Z}}\hat\Phi(\frac{\xi}{2^m}) = 1, \q\mbox{for } \ \xi\neq 0.
$$
Inserting this decomposition into $T_{j+j_l}$, we obtain
$$
 T_{j+j_l}(f,g)(x) =  \int \int \hat f(\xi) \hat g(\eta) e^{2\pi i (\xi+\eta)x} 
\sum\limits_{m\in\mathbb{Z}}\sum\limits_{n\in\mathbb{Z}}\mathcal{M}_{m,n}(\xi,\eta)
d\xi d\eta
 $$
where
$$
\mathcal{M}_{m,n}(\xi,\eta)=\hat\Phi(\frac{\xi}{2^{j_l+j+m}})\hat\Phi(\frac{\eta}{2^{j_l+lj+n}})\left( \int e^{-2\pi i \big(t\xi +(a_lt^l+P_l(t))\eta\big)}\rho_{j+j_l}(t)dt \right).
$$
Change variable $t\mapsto 2^{-j_l-j}t$ and let $Q_l(t)= 2^{j_l+lj}P_l(2^{-j_l-j}t)$. Notice that $|a_l|= 2^{(l-1)j_l} $, we have

\begin{align}\label{multiplier}
\mathcal{M}_{m,n}(\xi,\eta)=\hat\Phi(\frac{\xi}{2^{j_l+j+m}})\hat\Phi(\frac{\eta}{2^{j_l+lj+n}})\bigg( \int e^{-2\pi i \big(\frac{t\xi}{2^{j_l+j}} +\frac{(t^l+Q_l(t))\eta}{2^{j_l+lj}}\big)}\rho(t)dt \bigg).
\end{align}
Here $t^l+Q_l(t)$ is a small perturbation of $t^l$. 
The most difficult case is when $m\sim n>0$. For other cases, the standard Coifman-Meyer treatment (see \cite{c-m1}, \cite{gr}, \cite{ks}, \cite{c-m2}) 
allows us to reduce them to paraproducts.  In \cite {li2008U},  the uniform estimates of the paraproducts were established.
So we only present the details for the case $m\sim n>0$.  
For fixed $m$, the number of $n$ with $m\sim n$ is finite and hence without loss of generality we may assume that $m=n$.
Let 
$$
\mathcal{M}_{m}(\xi,\eta)=\hat\Phi(\frac{\xi}{2^{j_l+j+m}})\hat\Phi(\frac{\eta}{2^{j_l+lj+m}})\bigg( \int e^{-2\pi i (\frac{t\xi}{2^{j_l+lj}} +\frac{(t^l+Q_l(t))\eta}{2^{j_l+lj}})}\rho(t)dt \bigg)
$$
 and 
 \begin{align}\label{Tjm}
 T_{j,m}(f,g)(x) =  \int \int \hat f(\xi) \hat g(\eta) e^{2\pi i (\xi+\eta)x} 
\mathcal{M}_{m}(\xi,\eta)d\xi d\eta.
\end{align}
Our goal is to prove the following two propositions.
\begin{proposition}\label{L1}
For each $2\leq l\leq d$, there are constants $\epsilon>0$ and $C$ depending only on $d$ such that 
$$
\bigg\|\sum\limits_{j\in J_l^*(N)}T_{j,m}(f,g)\bigg\|_1\leq C2^{-\epsilon m}\|f\|_2\|g\|_2.
$$
\end{proposition}

\begin{proposition}\label{below1}
For $2\leq l \leq d$, $r>\frac{1}{2}$ and $p_1$, $p_2>1$ satisfying $ \frac{1}{p_1}+\frac{1}{p_2} = \frac{1}{r}$ there is a constant $C$ depending on $p_1, p_2$ and $d$ such that,
$$
\bigg\|\sum\limits_{j\in J_l^*(N)}T_{j,m}(f,g)\bigg\|_{r,\infty}\leq C m \|f\|_{p_1}\|g\|_{p_2}
$$
\end{proposition}
The proof of Proposition \ref{L1} is a modification of the sigma uniformity argument in \cite{li2008}. 
For Proposition \ref{below1} we employ the method of time frequency analysis. 
Sections \ref{section5}, \ref{rnot1} are devoted to prove Proposition \ref{L1} and Proposition \ref{below1} respectively. Interpolating Proposition \ref{L1} and Proposition \ref{below1}, we obtain a stronger version of Proposition \ref{below1}:
\begin{theorem}\label{below11}
For $2\leq l \leq d$, $r>\frac{1}{2}$ and $p_1$, $p_2>1$ satisfying $ \frac{1}{p_1}+\frac{1}{p_2} = \frac{1}{r}$ there is a constant $C$ depending on $p_1, p_2$ and $d$ such that,
$$
\bigg\|\sum\limits_{j\in J_l^*(N)}T_{j,m}(f,g)\bigg\|_{r}\leq C 2^{-\epsilon m} \|f\|_{p_1}\|g\|_{p_2}.
$$
\end{theorem}
Summing all the $m\geq 0$, combining the analysis in the beginning of this section, we obtain Theorem \ref{goodL}.


\section{The decay estimates from $L^2\times L^2$ to $L^1$}\label{section5}
\setcounter{equation}0

This section is devoted to the proof Proposition $\ref{L1}$. We adapt the method in \cite{li2008}.   
 We only present details for the case $j>0$ because the case $j<0$ can be treated similarly. Notice that 
$$
T_{j,m}(f,g)(x) = T_{j,m}(R_{j_l+j+m}(f),R_{j_l+lj+m}(g))(x),
$$
where 
$$
R_{j_l+j+m}(f)(x) = \int_{|\xi| \sim 2^{j_l+j+m}} \hat f(\xi)e^{2\pi i x\xi}d\xi 
$$
and 
$$
R_{j_l+lj+m}(g)(x)= \int_{|\xi| \sim 2^{j_l+lj+m}} \hat g(\xi)e^{2\pi i x\xi}d\xi. 
$$
By applying Cauchy-Schwarz Inequality and Plancherel's Theorem, to prove Proposition \ref{L1}, it is sufficient to prove the case of a single scale,
that is,  for every $j\in J_l^*(N)$, we only need to establish
\begin{align}\label{single}
\|T_{j,m}(f,g)\|_1\leq C2^{-\epsilon m}\|f\|_2\|g\|_2.
\end{align}
By rescaling variables: $\xi \mapsto 2^{j_l+j+m}\xi$, $\eta\mapsto 2^{j_l+lj+m}\eta$ and $x \mapsto 2^{-(j_l+lj+m)}x$, it is easy to see 
that (\ref{single}) is equivalent to  
\begin{align}\label{single2}
\|B_{j,m}(f,g)\|_1\leq C2^{-\epsilon m}\|f\|_2\|g\|_2,
\end{align}
where 
\begin{equation}
B_{j,m}(f,g)(x)
=2^{-(l-1)j} \int f\ast \Phi(2^{-(l-1)j}x-2^mt) g \ast \Phi(x-2^m(t^l+Q_l(t)))\rho(t)dt.
\end{equation}
The proof of (\ref{single2}) will be achieved by the following two propositions:
\begin{proposition}\label{000}
Let $j\geq N$ then there is a uniform constant $C$ such that
\begin{align}\label{small}
\|B_{j,m}(f,g) \|_1 \leq C 2^{\frac{(l-1)j-m}{6}} \|f\|_2\|g\|_2.
\end{align}
\end{proposition}
\begin{proposition}\label{001}
There are uniform constants $C$ and  $\epsilon'>0$ such that if $(l-1)j>(1-\epsilon')m $ then there is a small constant $\epsilon>0$ such that 
\begin{align}\label{big}
\|B_{j,m}(f,g) |\|_1 \leq C 2^{-\epsilon m} \|f\|_2\|g\|_2
\end{align}
holds for all $f, g\in L^2$. 

\end{proposition}
We introduce a notation in order to handle the perturbation term $Q_l(t)$. In the rest of this article, $D$ represents the differential operator and 
$\mathcal J$ represents $(1/2,2)$ or $(-2,-1/2)$.  
\begin{definition}\label{pair0}
Let $F$ be a smooth function defined on $\mathcal J$. Given a nonnegative integer $K$, we define the $\|\cdot \|_{D_K}$-norm of $F$ as
\begin{align*}
\|F\|_{D_K} = \sup_{0\leq k\leq K}\|D^kF\|_{L^\infty(\mathcal J)}.
\end{align*}
\end{definition}
Taking $K= 100d^2$ and $N$ to be sufficiently large, one can see immediately from the definition that  
\begin{equation}
\|Q_l\|_{D_K} \leq 2^{-99N/100}.
\end{equation}
\subsection{ Proof
 of Proposition \ref{000}}

The proof is based on a $TT^*$ method. 
Since $|t|\sim 1$, hence in the proof (\ref{small}) we may assume the support of $x$ is on an interval of length around $ 2^{(l-1)j+m}$. 
By stationary phase method, we have 
\begin{align}\label{phase}
\hat\Phi(\xi)\hat\Phi(\eta)\int e^{-2 \pi i 2^m\big(t\xi + (t^l+ Q_l(t))\eta\big)}\rho(t)dt \sim  2^{-m/2}e^{2^m i \phi_l(\xi,\eta)}:=2^{-m/2}\mathcal{N}_m (\xi,\eta),
\end{align}
where the phase $\phi_l(\xi,\eta)$ can be specified explicitly as follows: let $t_0$ be the root of
 $$
 \frac{d}{d t} \left(t\xi + (t^l+ Q_l(t))\eta\right) = 0\,.
 $$ 
Clearly $t_0$ is a function of $z=\xi/\eta$.  Then 
\begin{equation}\label{phase0}
\phi_l(\xi,\eta) := 2\pi \left (t_0 \frac{\xi}{\eta}+\left(t_0^l+Q_l(t_0)\right)\right)\eta\,.
\end{equation}

From (\ref{phase}),  $B_{j,m}(f, g)$ can be replaced by  
\begin{align*}
B'_{j,m}(f,g)(x)& =2^{\frac{-(l-1)j-m}{2}}\int \hat f(\xi)\hat g(\eta) \hat\Phi(\xi)\hat\Phi(\eta) e^{2\pi i (2^{-(l-1)j }\xi +\eta) x}e^{2^m i \phi_l(\xi,\eta)} d\xi d\eta.
\end{align*}
Let $h$ be a function supported on an interval of length around $2^{(l-1)j+m}$ and define the trilinear form

\begin{eqnarray*}
\Lambda_{j,m}(f,g,h) & : = & \langle B'_{j,m}(f,g),h\rangle
\\
&= & 
2^{\frac{-(l-1)j-m}{2}}\int \hat f(\xi)  \hat g(\eta) \hat\Phi(\xi) \hat\Phi(\eta) h (2^{-(l-1)j }\xi +\eta) \mathcal N_m(\xi,\eta) d\xi d\eta.
\end{eqnarray*}
Here $\mathcal N_m(\xi, \eta)= e^{i2^m \phi_l(\xi, \eta)}$. 
Change variables $\xi \mapsto \xi-\eta$ and $ \eta\mapsto  b_1\xi+b_2\eta$ with $b_1 =1-2^{-(l-1)j}$ and $b_2=2^{-(l-1)j}$,
and we write $\Lambda_{j,m}(f,g,h)$ as  
\begin{align*}
2^{\frac{-(l-1)j-m}{2}}\int \hat f(\xi-\eta)\hat g(b_\xi+b_2\eta) \hat\Phi(\xi-\eta) \hat\Phi(b_1\xi+ b_2\eta) h (\xi ) \mathcal N_m(\xi-\eta,b_1\eta+b_2\eta) d\xi d\eta.
\end{align*}

Invoking Cauchy-Schwarz inequality, we then control  $|\Lambda_{j,m}(f,g,h)|$ by 
\begin{align*}
2^{\frac{-(l-1)j-m}{2}}\| \bar B_{j,m}(f,g)\|_2 \|h\|_2,
\end{align*}
where
$$
\bar B_{j,m}(f,g)(\xi)=\int \hat f(\xi-\eta) \hat g(b_1\xi+b_2\eta) \hat\Phi(\xi-\eta)\hat\Phi(b_1\xi+ b_2\eta)  \mathcal N_m(\xi-\eta,b_1\eta+b_2\eta) d\eta.
$$
A direct calculation yields  
$$ \|\bar B_{j,m}(f,g)\|^2_2 = 
\int\left( \int \int F(\xi,\eta_1,\eta_2)G(\xi,\eta_1,\eta_2) \mathcal{K}_m(\xi,\eta_1,\eta_2)d\eta_1d\eta_2\right) d\xi,
$$
where
$$
F(\xi,\eta_1,\eta_2) = (\hat f\hat\Phi) (\xi-\eta_1)\overline{(\hat f\hat\Phi)(\xi-\eta_2)},
$$
$$
G(\xi,\eta_1,\eta_2) = (\hat g\hat\Phi) (b_1\xi+b_2\eta_1)\overline{(\hat g\hat\Phi)(b_1\xi+b_2\eta_2)}
$$
and
$$
\mathcal{K}_m(\xi,\eta_1,\eta_2) = \mathcal{N}_m (\xi-\eta_1,b_1\xi+b_2\eta_1)\overline {\mathcal{N}_m(\xi-\eta_2,b_1\xi+b_2\eta_2)}.
$$

Changing variables again: $\eta_1 \mapsto \eta$ and $\eta_2 \mapsto \eta+\tau$,  we represent $\|\bar B_{j,m}(f,g)\|^2_2$ as
$$
\int\left( \int \int F_\tau(\xi-\eta)G_\tau(b_1\xi+b_2\eta) \mathcal{K}_m(\xi,\eta,\eta+\tau)d\xi d\eta\right) d\tau,
$$
where 
$$
F_\tau(\cdot) = (\hat f\hat\Phi) (\cdot)\overline{(\hat f\hat\Phi)(\cdot -\tau)}
$$
and
$$
G_\tau(\cdot) = (\hat g\hat\Phi) (\cdot )\overline{(\hat g\hat\Phi)(\cdot -b_2\tau)}.
$$
Let $(u,v) = (\xi-\eta, b_1\xi +b_1\eta)$, the inner double integral becomes 
\begin{align}\label{00}
 \int \int F_\tau(u)G_\tau(v) e^{i2^m\mathcal Q_\tau(u,v)}du dv,
\end{align}
where 
$$
\mathcal Q_\tau(u,v) = \phi_l(u,v) - \phi_l(u-\tau,v+b_2\tau).
$$
We need the following lemma whose proof can be found in the appendix. 
\begin{lemma}\label{perturbation1}
Let $N$ be a sufficiently large number depending on $d$. 
For $j\geq N$ and $u,v, u-\tau, v+b_2\tau\in supp\, \hat\Phi$, there is a uniform constant $C>0$ such that
$$
|\partial_u\partial_v\mathcal{Q}_\tau(u,v)| \geq C|\tau|.
$$
\end{lemma}
Applying well-known H\"ormander's theorem \cite{ho}, we see that (\ref{00}) is estimated by
$$
C\min\{1, 2^{-m/2}|\tau |^{-1/2}\} \|F_\tau \|_2 \|G_\tau\|_2.
$$
This will provide us the bound 
$$
\|\bar B_{j,m}(f,g)\|_2 \leq C 2^{\frac{(l-1)j-m}{6}} \|f\|_2\|g\|_2.
$$
Therefore, for any function $h$ supported on an interval of length $2^{(l-1)j+m}$, we have 
$$
|\Lambda_{j,m}(f,g,h)| \leq 2^{\frac{-(l-1)j-m}{2}}2^{\frac{(l-1)j-m}{6}} \|f\|_2\|g\|_2\|h\|_2\leq 2^{\frac{(l-1)j-m}{6}} \|f\|_2\|g\|_2\|h\|_\infty,
$$
which completes the proof of Proposition \ref{000}.

%
%
%
%
%
%
%
%
%
%
%

\subsection{Proof of Proposition \ref{001}}

In order to prove Proposition \ref{001}, we introduce the concept of $\sigma$-uniformity. Let $\sigma\in (0,1]$, $\mathcal Q$ be a collection of real-valued measurable functions and $I$ a fixed bounded interval in $\mathbb R$. 
\begin{definition}\label{uniformity}
A function $f\in L^2(I)$ is $\sigma$-uniform in $\mathcal Q$ if 
\begin{align*}
|\int_I f(\xi)e^{-iq(\xi)}d\xi | \leq \sigma \|f\|_{L^2(I)}
\end{align*}
for all $q\in \mathcal Q$. Otherwise, $f$ is said to be $\sigma$-nonuniform in $\mathcal Q$. 
\end{definition}
The following theorem and its proof can be found in \cite{li2008}.
\begin{theorem}\label{uniform}
Let $L$ be a bounded sub-linear functional from $L^2(I)$ to $\mathbb C$, let $S_\sigma$ be the set of all functions that are $\sigma$-uniform in $\mathcal Q$, and let 
\begin{align*}
U_\sigma = \sup_{f\in S_\sigma}\frac{|L(f)|}{\|f\|_{L^2(I)}}
\end{align*}
Then for all functions $f\in L^2(I)$, 
\begin{align*}
|L(f)| \leq \max \{U_\sigma,2\sigma^{-1}Q\}\|f\|_{L^2(I)}
\end{align*}
where 
\begin{align*}
Q= \sup_{q\in \mathcal Q}|L(e^{iq})|.
\end{align*}
\end{theorem}

Unlike Proposition \ref{000}, the $TT^*$ method does not work at all for Proposition \ref{001}, as explained in \cite{li2008}. 
However, the $\sigma$-uniformity allows us to utilize $TT^*$ method in a subspace of $L^2$. This is the main reason why we 
introduce the concept.  
Let $K=100d^2$ and define 
\begin{align}
\mathcal Q_l =\Big\{a(\xi^\frac{l}{l-1}+\varepsilon (\xi))+b\xi: \, \frac{1}{2^{100}}\leq \frac{|a|}{2^m}\leq 2^{100}\!\!, \, b\in \mathbb R, 
 \|\varepsilon(\xi)\|_{D_K} \leq 2^{-\frac{N}{2}}\!\!,\, \xi \in supp\, \hat\Phi \Big\}.
\end{align}
First of all, we assume $\hat f $ has $\sigma$-uniformity in $\mathcal Q_l$. From the definition, we then have,   
for every $q\in \mathcal Q_l$, 
\begin{equation}
\left|\int \hat f(\xi)e^{-iq(\xi)}d\xi \right| \leq \sigma \|\hat f\|_2,
\end{equation}
here we have assumed $\hat f $ supports on  $\supp\hat\Phi$. 
Since $x$ is assumed to be supported on an interval of length around $2^{(l-1)j+m}$, we partition this interval into $C 2^{m}$ 
subintervals of length $\sim$ $2^{(l-1)j}$: $I_{k} =[\alpha_k -2^{(l-1)j}, \alpha_k+2^{(l-1)j}]$.
 Also notice that for fixed $x$, when $t$ varies on the support of $\rho$, $x-2^m(t^l+Q_l(t))$ ranges over an interval of length around $2^m$.
 Let $I_k' =[\alpha_k -C(2^{(l-1)j}+2^m), \alpha_k+C(2^{(l-1)j}+2^m)]$.  We choose the uniform constant $C$ (depending on $l$) 
to be sufficiently large such that 
$$
\Id_{I_k}(x)g \ast \Phi(x-2^m(t^l+Q_l(t))) =
\Id_{I_k}(x)(\Id_{I_k'}\, \cdot (g \ast \Phi))(x-2^m(t^l+Q_l(t))).
$$
Let $g_k(x) = (\Id_{I_k'}\,\cdot(g \ast \Phi))(x)$, then
\begin{align*}
B_{j,m}(f,g)(x)
&=2^{-(l-1)j/2} \int f\ast \Phi(2^{-(l-1)j}x-2^mt) g \ast \Phi(x-2^m(t^l+Q_l(t)))\rho(t)dt
\\
&=2^{-(l-1)j/2} \sum\limits_{k} \Id_{I_k}(x)\int \int\hat f(\xi) \hat\Phi(\xi) e^{2\pi i(2^{-(l-1)j}\xi + \eta)x} \hat g_k(\eta)
\\
&\q\q\q\q\q\q\q\q\q\q\cdot \left(\int e^{-2\pi i 2^m\left(t\xi+(t^l+Q_l(t))\eta\right)}\rho(t)dt\right) d\xi d\eta.
\end{align*}
We may insert the restriction condition $|\eta| \sim 1$ into the integrand,  because otherwise
 we apply integration by parts for 
$\int e^{-2\pi i 2^m\left(t\xi+(t^l+Q_l(t))\eta\right)}\rho(t)dt $
to obtain a decay factor $2^{-100m}$. Thus we may consider the major contribution $B''_{j,m}(f,g)(x)$ given by 
\begin{align*}
2^{-(l-1)j/2} \sum\limits_{k} \Id_{I_k}(x) \int \int\hat f(\xi) \hat\Phi(\xi) e^{2\pi i(2^{-(l-1)j}\xi + \eta)x} \hat g_k(\eta)\hat\Phi(\eta)
\\
\cdot \int e^{-2\pi i 2^m\left(t\xi+(t^l+Q_l(t))\eta\right)}\rho(t)dt d\xi d\eta.
\end{align*}
Taylor expansion allows us to represent $B''_{j,m}(f,g)(x)$ as
\begin{align*}
&2^{-(l-1)j/2}  \sum\limits_{p=0}^\infty 
 \frac{1}{p!}\sum\limits_{k} \left(\frac{2\pi i (x-\alpha_k)}{2^{(l-1)j}}\right)^p\Id_{I_k}(x)\cdot
\\
& \int \left( \int
\hat f(\xi) \hat\Phi(\xi)\xi^p  e^{2\pi i2^{-(l-1)j}\xi\alpha_k }\int e^{-2\pi i 2^m\left(t\xi+(t^l+Q_l(t))\eta\right)}\rho(t)dt d\xi\right) \hat g_k(\eta)\hat\Phi(\eta)e^{2\pi i \eta x}d\eta.
\\
\end{align*}
Pairing with $h$ and let $h_{l,k,p}(x) = \left(\frac{2\pi i (x-\alpha_k)}{2^{(l-1)j}}\right)^p\Id_{I_k}(x)h(x) $, we then have 
\begin{align}\label{mu}
\langle B''_{j,m}(f,g),h\rangle =2^{-(l-1)j/2}  \sum\limits_{p=0}^\infty 
 \frac{1}{p!}\sum\limits_{k} 
 \int \Gamma_{l,k}(\eta) \hat g_k(\eta)\check h_{l,k,p}(\eta)d\eta
\end{align}
where
\begin{align*}
\Gamma_{l,k}(\eta) 
&=
\hat\Phi(\eta) \int
\hat f(\xi) \hat\Phi(\xi)\xi^p  e^{2^{-(l-1)j}\xi\alpha_k }\int e^{-2\pi i 2^m\left(t\xi+(t^l+Q_l(t))\eta\right)}\rho(t)dt d\xi.
\end{align*}
Since $\alpha_k$ is the center of the interval $I_k$ with length $2\cdot 2^{(l-1)j}$, we majorize $h_{l,k,p}$ pointwise by
\begin{equation}\label{est-hlkp}
 \left| h_{l, k, p}(x) \right|\leq (2\pi)^p \Id_{I_k}(x) |h(x)|\,. 
\end{equation}

The method of stationary phase gives that the principal contribution of $\Gamma_{l,k}(\eta)$ is 
\begin{align*}
2^{-m/2}\hat\Phi(\eta) \int
\hat f(\xi) \hat\Phi(\xi)\xi^p  e^{2\pi i2^{-(l-1)j}\xi\alpha_k }e^{i2^m c_l\phi_l(\xi,\eta)} d\xi.
\end{align*}
Notice that $|\eta|\sim 1$ and hence $2^m c_l\phi_l(\xi,\eta)+b\xi \in \mathcal Q_l$.  By utilizing Fourier series and the assumption 
of $\sigma$-uniformity, the $L^\infty$-norm of the above expression is estimated by $C2^{-m/2}\sigma \|f\|_2$. This yields
\begin{align}\label{gammalk}
\|\Gamma_{l,k}\|_\infty  \leq C 2^{-m/2} \sigma \|f\|_2.
\end{align}

By applying Cauchy-Schwarz inequality, (\ref{est-hlkp}) and (\ref{gammalk}), (\ref{mu}) can be dominated by
\begin{align*}
& \,2^{-(l-1)j/2}  \sum\limits_{p=0}^\infty 
 \frac{1}{p!}\sum\limits_{k} 
 \| \Gamma_{l,k}\|_\infty \| \hat g_k\|_2\|\check h_{l,k,p}\|_2
\\
\leq
& \,2^{-(l-1)j/2-m/2}\sigma \sum\limits_{p=0}^\infty 
 \frac{1}{p!}\sum\limits_{k} \|f\|_2 \| g_k\|_2\| h_{l,k,p}\|_2
\\
\leq
& \,2^{-(l-1)j/2-m/2}\sigma \|f\|_2 \sum\limits_{p=0}^\infty 
 \frac{1}{p!}\left( \sum\limits_{k}\| g_k\|_2^2\right)^{1/2}\left( \sum\limits_{k}\|h_{l,k,p}\|^2_2\right)^{1/2}
\\
\leq
& \, C2^{-(l-1)j/2-m/2}\sigma \|f\|_2 \left( \sum\limits_{k}\| g_k\|_2^2\right)^{1/2}\|h\|_2.
\end{align*}
Notice that $2^{-(l-1)j/2-m/2}\|h\|_2 \leq \|h\|_\infty$ and the intervals $I_k'$'s overlap iff $m>|(l-1)j|$. 
Therefore if $\hat f$ has $\sigma$-uniformity in $\mathcal Q_l$, we have 
\begin{align}\label{sigma}
|\langle B''_{j,m}(f,g)(x),h(x) \rangle| \leq 
	\left\{ \begin{array}{rcl}
         C\sigma \|f\|_2\|g\|_2\|h\|_\infty & \mbox{if}
         & |(l-1)j|\geq m;
	\\
	\\  C2^{(m-(l-1)j)/2}\sigma \|f\|_2\|g\|_2\|h\|_\infty  & \mbox{if} & |(l-1)j|<m.
	                \end{array}\right.
\end{align}

It remains to estimate the case 
$\hat f(\xi) =e^{iq(\xi)}$
, where $q \in \mathcal Q_l $.
We need to estimate the absolute value of 
\begin{align*}\label{plugin}
\Lambda_q(g,h) = 2^{-(l-1)j/2}\int \int\left(\int \hat\Phi(\xi)e^{i(a\theta_l(\xi)+b\xi)}e^{2\pi i \xi (2^{-(l-1)j}x-2^m t)} d\xi\right)
\\
 g\ast\Phi(x-2^m(t^l+Q_l(t)))\rho(t)d t \, h(x)dx
\end{align*}
by $C2^{-\epsilon}\|g\|_2\|h\|_\infty$, for all $h$ supported on an interval of length around $2^{(l-1)j+m}$ and all 
$q(\xi) = a \theta_l(\xi) +b\xi$ with $\|\theta(\xi) -\xi^{\frac{l}{l-1}}\|_{D_K} \leq 2^{-N/2}$. 
By a standard rescaling argument via changing variables $x = 2^{(l-1)j+m}y$, it is equivalent to prove 
$|\Lambda'_q(g,h)| \leq C2^{-\epsilon m}\|g\|_2\|h\|_\infty$ for $h$ supported on an interval of length around 1, where
$\Lambda'_q(g,h)$ is defined as
\begin{equation}
  \iint \!\mathcal{P}(y,t)
 g\ast\Phi\Big(y-2^{-(l-1)j}(t^l+Q_l(t))\Big)\rho(t)d t \, h(y)dy.
\end{equation}
Here $\mathcal{P}(y,t)$ is given by 
\begin{align*}
\mathcal{P}(y,t)
&= 2^{m/2}\int \hat\Phi(\xi)e^{i\big(a\theta_l(\xi)+b\xi\big)}e^{2\pi i \xi (2^{m}y-2^mt)} d\xi
\\
&= 2^{m/2}\int \hat\Phi(\xi)e^{ia\big(\theta_l(\xi)+\frac{2\pi 2^{m}}{a}(y-t+b/2^m)\xi\big)} d\xi.
\\
\end{align*}
From the stationary phase method, we know that the principal part of $\mathcal{P}(y,t)$ is 
 \begin{align*} 
  \mathcal P'(y,t) = \left(\frac{2^m}{|a|}\right)^{1/2}e^{ia\beta\big(c(y-t+b')\big)}\hat\Phi\Big(c(\zeta(y-t+b'))\Big)
 \end{align*}
where $b'= b/2^m$, $c =  \frac{2\pi 2^{m}}{a}$, $\zeta(z)=(D\theta_l)^{-1}(-z)$ the solution of $D_\xi\left( \theta_l(\xi)+z\xi\right) = 0$ 
and $\beta(z) = \theta_l(\zeta(z))+z\zeta(z)$. Notice that $y$ and $t$ are both supported on intervals of lengths 
$\sim 1$, $2^{-200}\leq |c|\leq 2^{200}$ and $\|\zeta(z)- z^{l-1}\|_{D_{K-5}}\leq 2^{-N/3}$ (see Lemma \ref{lemmakey} in Appendix). 
Thus $\hat\Phi\big(c(\zeta(y-t+b'))\big)$ can be dropped. This can be done by expanding it into Fourier series. We omit the details. 
So it is enough for us to handle the principal part   
$$
\int\int e^{ia\beta(c(y-t+b'))} g(y-2^{-(l-1)j}(t^l+Q_l(t)))\rho(t)d t \, h(y)dy.
$$ 
Change variable $s = t^l +Q_l(t)$ and let $\kappa_l(s)$ denote the inverse function of $t^l +Q_l(t)$. Let $\tilde\rho(s) =\rho(t)$, 
 then we just need to dominate the following
\begin{align*}
\Lambda_q''(g,h)
&= \int\int e^{ia\beta(c(y-\kappa_l(s)+b'))} g(y-2^{-(l-1)j}s)\tilde\rho(s)d s \, h(y)dy
\\
&=
 \int\int e^{ia\beta(c(y+2^{-(l-1)j}s-\kappa_l(s)+b'))} h(y+2^{-(l-1)j}s)\tilde\rho(s)ds \, g(y)dy,
\end{align*}
which is bounded by $\|g\|_2\|\mathcal T(h) \|_2$, where 
\begin{align}
\mathcal T(h)(y) = \int e^{ia\beta(c(y+2^{-(l-1)j}s-\kappa_l(s)+b'))} h(y+2^{-(l-1)j}s)\tilde\rho(s)ds.
\end{align}
By employing the $TT^*$ method again, $\|\mathcal T(h)\|_2^2$ equals
\begin{align*}
\int \left(\int \int e^{ia\varphi_\tau(y,t)}H_\tau(y+2^{-(l-1)j}s)\Theta_\tau(s) dy ds\right) d\tau
\end{align*}
where 
\begin{align*}
&H_\tau(y) = h(y)\overline {h(y+2^{-(l-1)j}\tau)},
\\
& \Theta_\tau(s)=\tilde\rho(s)\tilde\rho(s+\tau),
\\
& \varphi_\tau(y,s) =\beta(c(y+2^{-(l-1)j}s-\kappa_l(s)+b'))-\beta(c(y+2^{-(l-1)j}(s+\tau)-\kappa_l(s+\tau)+b')).
\end{align*}
Changing coordinates $(y,t)\mapsto (u,v)$ by $u=y+2^{-(l-1)j}s$ and $v=s$, we write the inner double-integral in the previous integral as  
\begin{align*}
\int \int e^{ia\mathcal O_\tau(u,v)}H_\tau(u)\Theta_\tau(v) du dv,
\end{align*}
where 
$$
\mathcal O_\tau(u,v) = \beta(c(u-\kappa_l(v)+b')) - \beta(c(u+2^{-(l-1)j}\tau-\kappa_l(v+\tau)+b')).
$$
In order to apply the operator version of van der Corput lemma (Lemma \ref{2dvan}), we need to control the lower bound of the derivatives of $\mathcal O_\tau(u,v)$:
\begin{lemma}\label{derivative00}
For $l\geq 2$ we have 
$$
|\partial^{l-1}_u\partial_v(\mathcal O_\tau(u,v))| \geq C_l|\tau|.
$$
Moreover, if $l=2$ we have 
$$
|\partial_u\partial^2_v(\mathcal O_\tau(u,v))| \geq C|\tau|.
$$
\end{lemma}
The proof of Lemma \ref{derivative00} can be found in the Appendix. 
\begin{lemma}\label{2dvan}
For fixed two bounded intervals $I_1$ and $I_2$, let $k\geq 1$ be integer and $\psi$ be a real function defined on $I_1\times I_2$ such that  
$$
|\partial_{x}^k\partial_{y} \psi(x,y)| \geq 1 \q\q\mbox{for all } (x,y)\in I_1\times I_2
$$
For the case $k=1$, assume that an additional (convexity) condition 
$$
|\partial_{x}^2\partial_{y} \psi(x,y)| \neq 0
$$
holds for any $(x,y)\in I_1\times I_2$. 
Then there are constants $\varepsilon=\varepsilon(k)>0$ and  $C$ depending only on the lengths of $I_1$ and $I_2$ such that  
\begin{align}\label{2vander}
\left |\int\int_{I_1\times I_2} e^{i\lambda\psi(x,y)} f(x)g(y)dxdy\right | \leq C(1+|\lambda|)^{-\varepsilon} \| f \|_2 \|g\|_2
\end{align}
for all $f$, $g$ $\in L^2$. 
\end{lemma}

The proof can be found in \cite{ca} for $k\geq 2$ and \cite{ph} for $k=1$. The above two lemmas yield: 
\begin{align}\label{some}
\| \mathcal T(h) \|^2_2  \leq C\int_{-10}^{10}\min\{1,2^{-\varepsilon a}\tau^{-\varepsilon}\} \|H_\tau \|_2  \|\Theta_\tau \|_2d\tau 
\leq C2^{-\varepsilon m/2} \| h \|^2_\infty.
\end{align}
By (\ref{sigma}), (\ref{some}) and Theorem \ref{uniform}, we obtain Proposition \ref{001}.

%
%
%
%
%
%
%
%
%
%
%

\section{The  $J_{\rm bad}(N)$ case}\label{rnot1}
\setcounter{equation}0

In this section, we employ the method of  time frequency analysis to prove Proposition \ref{below1} for a fixed $l$. The proof is technical 
and in order not to sidetrack the readers from the main story, we decide to put the proofs of technical lemmas to the later sections
and only state those lemmas in this section.  Most of those lemmata are variances of the known results related to Littlewood-Paley theory and $BMO$ theory. 
Denote  
$$
T^l_m(f,g)(x)= \sum\limits_{j\in J_l^*(N)}T_{j,m}(f,g)(x).
$$
Written as a convolution, $T^l_{j,m}(f,g)(x)$ equals
$$
\sum\limits_{j\in J_l^*(N)}\int f\ast \Phi_{j_l+j+m}\Big(x-2^{-j_l-j}t\Big)g\ast\Phi_{j_l+lj+m}\Big(x-2^{-j_l-lj}(t^l+Q_l(t))\Big)\rho(t)dt,
$$
where $\Phi_k$ is a Schwartz function whose Fourier transform is $\hat\Phi_{k}(\xi) = \hat\Phi(\frac{\xi}{2^k})$. We consider a larger operator, by putting the absolute value inside the integral. More precisely, we define $|T|^l_m(f,g)(x)$ as
$$
\sum\limits_{j\in J_l^*(N)}\int \Big|f\ast \Phi_{j_l+j+m}\Big(x-2^{-j_l-j}t\Big)g\ast\Phi_{j_l+lj+m}\Big(x-2^{-j_l-lj}(t^l+Q_l(t))\Big)\rho(t)\Big|dt.
$$
Since the cancellation condition of $\rho$ does not play any role in the rest of the proof, we can safely assume $\rho\geq 0$. The reader may also notice that, if we sum over all the $l$'s of the above operator, the resulting sum has a pointwise control over the maximal function, which paces the way to dominate the bilinear maximal function in Section \ref{maxfunction}. We invoke the following well-known lemma concerning the characterization of weak norm $L^{p,\infty}$. The lemma and its proof can be found in \cite{au} or \cite{cm}. 
\begin{lemma}\label{weak}
Let $0<p<\infty$ and $A>0$. Then the following statements are equivalent:
\\
(i) $\|f\|_{p,\infty}\leq A$
\\
(ii) For every Lebesgue measurable set $E$ with $0<|E|<\infty$, there is a subset $E'\subset E$ with $|E'|\leq \frac{|E|}{2}$ such that 
$|\langle f,\Id_{E'} \rangle |\leq CA|E|^{\frac{1}{p'}}$.  Here $C$ is an absolute constant and $p'$ is defined by $\frac{1}{p}+\frac{1}{p'} = 1$. ($p'$ can be negative!)  
\end{lemma}
In proving the weak norm, we can assume $f$ and $g$ are both characteristic functions, i.e. $f(x) = \Id_{F_1}(x)$ and $g(x)=  \Id_{F_2}(x)$ through out this section, where $ 0<|F_1|, |F_2|<\infty$.  Consider the exceptional set 
\begin{equation}\label{defofExp}
\Omega = \Big\{x: M\Id_{F_1}(x)> C|F_1|/|F_3|\Big\}\bigcup \Big\{x: M\Id_{F_2}(x)> C|F_2|/|F_3|\Big\},
\end{equation}
where $M$ is uncentered Hardy-Littlewood maximal function. Choosing $C$ large enough, we have $|\Omega |< |F_3|/2$. Thus $F_3' = F_3\backslash \Omega$ has Lebesgue measure strictly greater than $|F_3|/2$. Our goal is to prove the following Proposition:
\begin{proposition}\label{progoal}   
There is a uniform constant $C$ satisfying 
\begin{align}\label{goalui}
|\langle|T|^l_m(f,g),\Id_{F_3'}\rangle|\leq Cm |F_1|^{1/p_1}|F_2|^{1/p_2}|F_3|^{1/r'}
\end{align}
where $r>1/2$, $p_1, p_2>1$,  and $1/p_1+1/p_2=1/r$. 
\end{proposition}
One can see that Proposition \ref{progoal} and Lemma \ref{weak} together yield Proposition \ref{below1}. We begin our proof of Proposition \ref{progoal} with removing some error terms. 
\subsection{Error Terms related to the exceptional set}\q\\
For $k\in \mathbb R$, let 
$\tilde\psi_k(x)=2^{k}\tilde\psi(2^kx)$, $\Omega_k = \{x\in \Omega: {\rm dist}(x, \Omega^c)\geq 2^{-k} \}$
and 
$\psi_k(x) = (\Id_{\Omega_k^c}*\tilde\psi_k)(x)$. 
Here $\tilde\psi$ is a (non-negative) Schwartz function such that $\wh{\tilde\psi}(0)=1$ and 
$ \wh{\tilde\psi}$ is supported on $[-1, 1]$. Let 
\begin{align*}
&F_{l,m,j} (x,t) = \psi_{j_l+j+m}f\ast \Phi_{j_l+j+m}(x-2^{-j_l-j}t),
\\
&F^c_{l,m,j} (x,t) = (1-\psi_{j_l+j+m})f\ast \Phi_{j_l+j+m}(x-2^{-j_l-j}t),
\end{align*}
\mbox{and}
\begin{align*}
&G_{l,m,j}(x,t) = \psi_{j_l+lj+m}g\ast\Phi_{j_l+lj+m}(x-2^{-j_l-lj}(t^l+Q_l(t))),
\\
&G^c_{l,m,j}(x,t) = (1-\psi_{j_l+lj+m})g\ast\Phi_{j_l+lj+m}(x-2^{-j_l-lj}(t^l+Q_l(t))).
\end{align*}
We call $F_{l,m,j} $, $G_{l,m,j}$ the major terms and $F^c_{l,m,j} $, $G^c_{l,m,j}$ the error terms. Moreover, we also set 
$$F_{l,m,j} (x) := F_{l,m,j} (x,0)$$
and 
$$G_{l,m,j} (x) := G_{l,m,j} (x,0).$$


\begin{lemma}\label{error}
Let $L_{1,m}(f, g)$ and $L_{2,m}(f, g)$ be given by

\begin{align*}
&L_{1,m}(f, g)= \sum_{j\in J_l^*(N)} \int_{\Omega^c} \int \left|F^c_{l,m,j} (x,t) G_{l,m,j} (x,t)\rho (t)\right|dt dx
\end{align*}
and 
\begin{align*}
&L_{2,m}(f, g)= \sum_{j\in J_l^*(N)} \int_{\Omega^c} \int \left|F_{l,m,j} (x,t) G^c_{l,m,j} (x,t)\rho (t)\right|dt dx.
\end{align*}
Then there exists a constant $C$ independent of $F_1, F_2, m$ such that
\begin{align}\label{error1}
 \left| L_{1,m}(f, g)\right| \leq Cm\left( \frac{|F_1|}{|F_3|}\right)^{1/p_1}\left( \frac{|F_2|}{|F_3|}\right)^{1/p_2}|F_3|
\end{align}
and
\begin{align}\label{error2}
 \left| L_{2,m}(f, g)\right| \leq Cm\left( \frac{|F_1|}{|F_3|}\right)^{1/p_1}\left( \frac{|F_2|}{|F_3|}\right)^{1/p_2}|F_3|
\end{align}
for any $p_1, p_2>1$. 
\end{lemma}

This lemma is one of the major difficulties in the proof of Proposition \ref{below1} and the key idea in our proof is
based on a use of  the Whitney decomposition.
We postpone its proof to Section \ref{P_error}.  Lemma {\ref{error}} gives the desired estimates for the error terms.
Henceforth, in the rest of this section, we focus only on the major term: 
\begin{align}\label{major03}
|T'|^l_m(f,g)(x)= \sum\limits_{j\in J_l^*(N)}\int |F_{l,m,j} (x,t)G_{l,m,j}(x,t)\rho(t)|dt.
\end{align}


It remains to prove the following two propositions. 
\begin{proposition}\label{major0} 
If $p_1$, $p_2$, $r >1$ and $\frac{1}{p_1}+\frac{1}{p_2}=\frac{1}{r}$ then
\begin{align*}
\Lambda(f,g):=|\langle |T'|^l_m(f,g),\Id_{F_3'}\rangle| \leq C|F_1|^{1/p_1}|F_2|^{1/p_2}|F_3|^{1-1/r}.
\end{align*}

\end{proposition}

\begin{proposition}\label{major1} 
If $\frac{1}{2}<r<1$, $1<p_1$, $p_2 <2$ and $\frac{1}{p_1}+\frac{1}{p_2}=\frac{1}{r}$ then
\begin{align*}
\Lambda(f,g):=|\langle|T'|^l_m(f,g),\Id_{F_3'}\rangle| \leq Cm|F_1|^{1/p_1}|F_2|^{1/p_2}|F_3|^{1-1/r}.
\end{align*}

\end{proposition}
It is clear  that Proposition \ref{below1} follows from Lemma \ref{error}, Proposition \ref{major0}, Proposition \ref{major1} and interpolation. 

\subsection{Proof of Proposition \ref{major0} }\q\q\q\\

Let 
$$
J^{*}_{l,+}(N) =J_l^*(N)\cap \{j>0\}\,\,\,{\rm and}\,\,\,
J^{*}_{l,-}(N) =J_l^*(N)\cap \{j<0\}.
$$

$|T'|^l_m(f,g)$ can be then written as the sum of
 \begin{equation}
|T_+'|^l_m(f,g)(x)= \sum\limits_{j\in J^*_{l,+}(N)}\int 
|F_{l,m,j}(x,t) G_{l,m,j}(x,t)
 \rho(t)|dt
\end{equation}
and
 \begin{equation}
|T_-'|^l_m(f,g)(x)= \sum\limits_{j\in J^*_{l,-}(N)}\int 
|F_{l,m,j}(x,t) G_{l,m,j}(x,t)
 \rho(t)|dt.
\end{equation}
Set
$$
\Lambda^{+}(f,g):= \left|\langle |T_+'|^l_m(f,g),\Id_{F_3'}\rangle\right| 
$$
and 
$$
\,\,\Lambda^{-}(f,g)\,:=\left|\langle|T_-'|^l_m(f,g),\Id_{F_3'}\rangle\right|.  
$$
Thus, it suffices to prove:
\begin{equation}\label{Lam+}
\Lambda^{+}(f,g) \leq C|F_1|^{1/p_1}|F_2|^{1/p_2}|F_3|^{1-1/r}
\end{equation}
and 
\begin{equation}\label{Lam-}
\Lambda^{-}(f,g) \leq C|F_1|^{1/p_1}|F_2|^{1/p_2}|F_3|^{1-1/r}.
\end{equation}

In the proof of Proposition \ref{major0} (also Proposition \ref{major1}), we consider only $j>0$ and present a proof for (\ref{Lam+}).
The other case $j<0$ is similar and thus (\ref{Lam-}) can be treated exactly same as (\ref{Lam+}). 
 Applying Fubini's Theorem to change the order of $dt$ and $dx$,  performing change of variable $u = x-2^{-j_l-lj}(t^l+Q_l(t))$ and finally 
 using  Fubini again,  we then obtain that $\Lambda^{+}(f,g)$ is bounded by
\begin{align*}
 \sum\limits_{j\in J^*_{l,+}(N)}\int
\int \left| F_{l,m,j}(u-\textbf{tr}_j(t))\rho(t)\right|dt  \,|G_{l,m,j}(u)| du,
\end{align*}
where $$\textbf{tr}_j(t) = 2^{-j_l-j}t-2^{-j_l-lj}(t^l+Q_l(t)).$$ 
Notice that 
\begin{align*}
\int \left| F_{l,m,j}(u-\textbf{tr}_j(t))\rho(t)\right|dt    \leq &C MF_{l,m,j}(u)
\end{align*}
and 
$$
\sup_{j\in J^*_l(N)} |\psi_{j_l+lj+m}(u)| \leq C M\Id_{\Omega}(u).
$$
Utilizing Cauchy-Schwarz Inequality, we get 
\begin{multline}
\Lambda^{+}(f,g)\leq 
C\int  \bigg( \sum\limits_{j\in J^*_{l,+}(N)} 
|MF_{j_l+j+m}|^2(u)\bigg)^{{1/2}}
\\
\bigg( \sum\limits_{j\in J^*_{l,+}(N)}(g\ast\Phi_{j_l+lj+m})^{2}(u)\bigg)^{1/2} M\Id_{\Omega}(u) du
\end{multline}
By H\"older's Inequality,  $\Lambda^{+}(f,g)$ is controlled by (up to a constant) 
\begin{align*}
\Big\| \Big( \sum\limits_{j\in J^*_{l,+}(N)}  |MF_{j_l+j+m}|^2\Big)^{{1/2}} \Big\|_{p_1}&
\Big\|\Big( \sum\limits_{j\in J^*_{l,+}(N)}(g\ast\Phi_{j_l+lj+m})^{2}\Big)^{1/2}\Big\|_{p_2}
\Big\|M\Id_{\Omega}(u)\Big\|_{r'}.
\end{align*}
Via Fefferman-Stein Inequality \cite{fs}, Littlewood-Paley Theory and Hardy-Littlewood Inequality, the first term is dominated by $C\|f\|_{p_1}$, the second one by $C\|g\|_{p_2}$ and the last one by $C\|\Id_\Omega\|_{r'}$. Therefore
\begin{align*}
\Lambda^{+}(f,g) \leq C |F_1|^{1/p_1}|F_2|^{1/p_2}|F_3|^{1/r'},
\end{align*}
as desired. This finishes the proof of Proposition \ref{major0}.\\

\subsection{Partition of unity and Trees}\q\q\q\\

The case $\frac{1}{2}< r <1$ is  much more complicated and the rest of this section is devoted to its proof. 
Again, we only handle the case $j>0$ without loss of generality. 
First we localize the variable $x$ by introducing a partition of unity. Indeed, 
let $\theta$ be nonnegative Schwartz function such that $\hat\theta(0)=1$ and $\hat\theta(\xi)$ is supported on $(-2^{-10},2^{-10})$. $\theta_{k}(x) = 2^k\theta(2^kx)$ for $k\in \mathbb{R}$. Let $I_{n,l,j} =[n \,2^{-j_l-j}, (n+1)\,2^{-j_l-j}]$ and 
$$\Id_{n,l,j}^*(x) =  \Id_{I_{n,l,j}}\ast \theta_{j_l+j+m}(x).$$
According to the definition of $\theta$ we have
$$
\Id_{\mathbb{R}}(x) = \sum\limits_{n\in \mathbb{Z}}\Id_{n,j,l}^*(x).
$$
Set 
$$
F_{n,l,m,j}(x,t) = \Id^*_{n,l,j}(x-2^{-j_l-j}t)F_{l,m,j}(x,t) 
$$
and
$$
G_{n,l,m,j}(x,t) =\Id^*_{n,l,j}(x-2^{-j_l-lj}(t^l+Q_l(t))) G_{l,m,j}(x,t). 
$$
Then we have the representation
\begin{multline}
|T_+'|^l_m(f,g)(x)\!= \!\!\!\!\!\!\!\sum\limits_{j\in J^*_{l,+}(N)}\int 
\left| \left(\sum\limits_{n\in\mathbb{Z}}F_{n,l,m,j}(x,t)\right)
\left(\sum\limits_{n\in\mathbb{Z}} G_{n,l,m,j}(x,t)\right)
 \rho(t)\right|dt
\end{multline}
The next step is to define a tree structure and in every single tree, remove effect of the translation ${\bf tr}_j(t) $ by 
utilizing Fefferman-Stein's inequality. 
\begin{definition}
Let 
$$S_0= \{ (j, n):  n\in\mathbb{Z},\, j\in J^*_{l,+}(N) \}.$$
A subset $T\subset S_0$ is called a tree if there is a $(j_0, n_0)\in S_0$ such that 
\begin{align}\label{top}
  I_{n,l,j} \subseteq I_{n_0,l,j_0} \,\,{\rm for}\,\,{\rm all}\,\, (j, n)\in T .
 \end{align}
\end{definition} 
 
We call $(j_0, n_0)$ the top of the tree $T$ if $I_{n_0,l,j_0}$ is the minimal interval of type $I_{n,l,j}$ such that (\ref{top}) holds,  
and denote $I_{n_0, l, j_0}$ by $ I_T$. 
$(j_0, n_0)$ is not required to be a member of $T$. For a subset $S\subset S_0$,  let $S_j $ denote $\{ n\in\mathbb{Z}: (j,n)\in S   \}$ and 
thus $n\in S_j$ iff $(j,n)\in S$. Set
\begin{align}\label{alltiles}
\Lambda_S(f,g)= \!\!\!\sum\limits_{j\in J^*_{l,+}(N)}\iint 
\Big|\Big(\sum\limits_{n\in S_j }F_{n,l,m,j}(x,t)\Big)
\Big(\sum\limits_{n\in S_j} G_{n,l,m,j}(x,t)\Big)
 \rho(t)\Big|dt dx.
\end{align}
Since $\Lambda_{S_0}(f,g)\geq \Lambda^+(f,g)$,   it is sufficient to prove
\begin{align}\label{Lambdagoal}
\Lambda_{S_0}(f,g) \leq C m |F_1|^{1/p_1}|F_2|^{1/p_2}|F_3|^{1-1/r}.
\end{align}
Now, we get rid of the effect of  translation caused by ${\bf tr}_j$ 
 in $\Lambda_T(f,g)$. Applying Fubini's Theorem to change the order of $dt$ and $dx$, then setting $u = x-2^{-j_l-lj}(t^l+Q_l(t))$ and finally 
 using Fubini again,  we have 
\begin{multline}
\Lambda_T(f,g) \!\!=\!\!\!\!\! \sum\limits_{j\in J^*_{l,+}(N)}\iint
 \Big|\sum\limits_{n\in T_j }{f}_{n,l,m,j}(u-\textbf{tr}_j(t))\rho(t)\Big|dt
 \Big|\sum\limits_{n\in T_j} g_{n,l,m,j}(u)\Big| du,
\end{multline}
where 
$$
f_{n,l,m,j}(x) =  \Id_{n,l,j}^*\psi_{j_l+j+m}f\ast\Phi_{j_l+j+m}(x)
$$
and
$$
g_{n,l,m,j}(x) =  \Id_{n,l,j}^*\psi_{j_l+lj+m}g\ast\Phi_{j_l+lj+m}(x).
$$
From  $j\in J^*_{l,+}(N)$ and $|t|\in (1/2,2)$,  we conclude 
\begin{equation}
2^{-j_l-j}|t| \geq 2^N2^{-j_l-lj}(t^l+Q_l(t)).
\end{equation}
Hence
$$
\int \Big|\sum\limits_{n\in T_j }f_{n,l,m,j}(x-\textbf{tr}_j(t))\rho(t)\Big|dt \leq CM(\sum\limits_{n\in T_j }f_{n,l,m,j})(x).
$$
Utilizing Cauchy-Schwarz inequality, we estimate $\Lambda_T(f,g)$ by 
\begin{align*}
C\bigg\|\bigg(\sum\limits_{j\in J^*_{l,+}(N)}\Big(M(\sum\limits_{n\in T_j }f_{n,l,m,j})(x)\Big)^2\bigg)^{1/2}
\bigg(\sum\limits_{j\in J^*_{l,+}(N)}\Big(\sum\limits_{n\in T_j }g_{n,l,m,j}(x)\Big)^2\bigg)^{1/2}\bigg\|_{L^1(\mathbb{R})}.
\end{align*}
By a use of H\"older's inequality,  $\Lambda_T(f,g)$ is bounded by
$$
C \bigg\|     \bigg(\sum\limits_{j\in J^*_{l,+}(N)}\Big(M(\sum\limits_{n\in T_j }f_{n,l,m,j})(x)\Big)^2\bigg)^{1/2}         \bigg\|_{p'_2}
\bigg\|\bigg(\sum\limits_{j\in J^*_{l,+}(N)}\Big(\sum\limits_{n\in T_j }g_{n,l,m,j}(x)\Big)^2\bigg)^{1/2}\bigg\|_{p_2}.
$$
Finally, Fefferman-Stein's Inequality yields
\begin{align}\label{***}
\Lambda_T(f,g) \leq C   \|S_{1,T}(f)\|_{p'_2}\|S_{2,T}(g)\|_{p_2},
\end{align}
where $S_{1, T}$ and $ S_{2, T}$ are square functions defined  by, respectively, 
\begin{equation}\label{defofSq1}
S_{1,T}(f)= \bigg(\sum\limits_{j\in J^*_{l,+}(N)}\bigg|\sum\limits_{n\in T_j }f_{n,l,m,j}(x)\bigg|^2\bigg)^{1/2}
\end{equation}
and
\begin{equation}\label{defofSq2}
S_{2,T}(g) = \bigg(\sum\limits_{j\in J^*_{l,+}(N)}\bigg|\sum\limits_{n\in T_j }g_{n,l,m,j}(x)\bigg|^2\bigg)^{1/2}.
\end{equation}

Although the definitions above are given for a tree $T$, we can extend them to the general square functions 
$S_{1, Z}$ and $S_{2, Z}$,  replacing  $T$ by any subset $Z$ of $S_0$ in (\ref{defofSq1}) and (\ref{defofSq2}), respectively.

\subsection{Sizes and BMO estimates}\q\q\q\\

%
%
%
%
%
%
%
%
%
Due to certain technical reasons, we need to define the sizes of a tree carefully. We begin with the previous definition of $\Id_{n,l,j}^*(x)$. Since $\theta$ is Schwartz function, for all positive integer $K$, there is an positive number $C_K$ depends only on $K$ and $\theta$ such that 
$$
|\theta_k(x)| \leq C_K\frac{2^k}{(1+2^k|x|)^K}.
$$
Thus
$$
\Id^*_{n,l,j}(x) \leq C_K \int_{I_{n,l,j}}\frac{2^{j_l+j+m}}{(1+2^{j_l+j+m}|x-y|)^{K}}dy
$$
and  
$$
\Id^*_{n,l,j}(x) \leq \frac{C_K} {(1+2^{j_l+j+m}{\rm dist}(x,I_{n,l,j}))^{K}}.
$$
Both upper bounds of $\Id^*_{n,l,j}(x) $ above will be useful. In what follows, the value of $K$ may be varied but it is completely harmless because all those $K$'s are absolute constants independent of the coefficients of $P$. Set
$$
\Id^{**}_{n,l,j}(x) = \int_{I_{n,l,j}}\frac{2^{j_l+j+m}}{(1+2^{j_l+j+m}|x-y|)^{K}}dy.
$$ 
After the above preparation, we are ready to define the sizes of a tree or a subset of $S_0$.   
\begin{definition}\label{size10} 
Let $T\subset S$ be a tree, the $1$-$ {  {\rm \bf {size}}}(T)$ and $2$-${ {{\rm{\bf size}}}}(T)$ are defined as
\begin{eqnarray*} 
 1\mbox{-}{\rm{\bf size}}  (T)   & =   & |I_T|^{-1/p_1} 
 \Bigg(  \bigg\|\bigg(\sum\limits_{(j,n)\in T}  \big| \Id^{**}_{n,l,j}\psi_{j_l+j+m}f\ast\Phi_{j_l+j+m}\big|^2\bigg)^{1/2}\bigg\|_{p_1}
  \\
 &  &  + \, \bigg\|\bigg(\sum\limits_{(j,n)\in T}\big|\Id^{**}_{n,l,j}\psi_{j_l+j+m}f\ast(D\Phi)_{j_l+j+m}\big|^2\bigg)^{1/2}\bigg\|_{p_1}
 \\
  & &+\,
 \bigg\|\bigg(\sum\limits_{(j,n)\in T}\big|\Id^{**}_{n,l,j}(D\psi)_{j_l+j+m}f\ast\Phi_{j_l+j+m}\big|^2\bigg)^{1/2}\bigg\|_{p_1}
 \Bigg)
\end{eqnarray*}
and 
\begin{flalign*}
 &\q\q2\mbox{-}{\rm{\bf size}}(T) =  |I_T|^{-1/p_2} 
 \bigg\|\bigg(\sum\limits_{(j,n)\in T}\big|\Id^{**}_{n,l,j}\psi_{j_l+lj+m}g\ast\Phi_{j_l+lj+m}\big|^2\bigg)^{1/2}\bigg\|_{p_2}.&
\end{flalign*}
Given an arbitrary subset $U\subset S_0$, the ${\rm{\bf size}}_1(U)$ and ${\rm{\bf size}}_2(U)$ are defined as 
\begin{equation}
{\rm{\bf size}}_1(U) = \sup\limits_{T\subset U} | 1\mbox{-}{\rm{\bf size}}(T)|
\end{equation}
and 
\begin{equation}
 {\rm{\bf size}}_2(U) = \sup\limits_{T\subset U} | 2\mbox{-}{\rm{\bf size}}(T)|,
 \end{equation}
 where $T\subset U$ is a tree.
\end{definition}

In the definition of $1$-${\rm{\bf size}}$, each of the summand is almost identical, the reason we make three "almost" identical copies is purely technical. The following inequalities  are consequences from the definition:
\begin{equation}\label{size111}
\|S_{1,T}(f)\|_{p_1} \leq {\rm{\bf size}}_1(T) \cdot |I_T|^{1/p_1}
\end{equation}
and 
\begin{equation}
\|S_{2,T}(g)\|_{p_2}\leq {\rm{\bf size}}_2(T)\cdot |I_T|^{1/p_2}.
\end{equation}
The sizes are also closely related to the maximal function:
\begin{lemma}\label{size1_Max}
For any tree $T \subset S$, we have the following estimates:
\begin{align}\label{1size}
1\mbox{-}{{\rm{\bf size}}}(T) \leq C \inf\limits_{x\in I_T} M_{p_1}f(x),
\end{align}  
\begin{align}\label{size1}
{\rm{\bf size}}_1(T) \leq C\min \{1, |F_1|/|F_3|\}^{1/p_1}
\end{align}  
\begin{equation}\label{2size}
2\mbox{-}{\rm{\bf size}}(T) \leq C\inf\limits_{x\in I_T} M_{p_2}g(x),
\end{equation}  
\begin{equation}\label{size2}
{\rm{\bf size}}_2(T) \leq C\min \{1, |F_2|/|F_3|\}^{1/p_2}.
\end{equation}  
\end{lemma}

The proof of Lemma \ref{size1_Max} is standard. It can be done by using $L^p$ estimate on the square function, 
a consequence of Littlewood-Paley theory.  Since $(j,n)\in T$,  the major contribution is from the function $f$ restricted to $2I_T$.
From this observation,  (\ref{1size}) and (\ref{2size}) follow.  (\ref{size1}) and (\ref{size2}) can be concluded by using (\ref{1size}), 
(\ref{2size}) and the definition of the exceptional set $\Omega$.  We omit the details.


We need two BMO estimates in order to sum $\Lambda_T(f,g)$ over all trees. The first one is a global control of 
$\|S_{1, Z}(f)\|_{BMO}$. Here $S_{1, Z}$ is the square function associated to a subset $Z$. 
In Lemma \ref{BMO1}, the tree structure is unimportant and we have the result for general subset $Z$. 
\begin{lemma}\label{BMO1}
 For $j\in J^*_{l,+}(N)$,  let $Z_j\in \mathbb{Z}$ be an arbitrary subset and any positive number $1<p<\infty$, there is a constant $C_p$ depending on $d$ and $p$ such that 
\begin{equation}\label{lp}
\|S_{1,Z}(f)\|_p \leq C_p \|f\|_p
\end{equation} 
and 
\begin{equation}\label{pBMO}
\|S_{1,Z}(f)\|_{BMO}  \leq C_p \min\{1, |F_1|/|F_3|\}^{1/p}.
\end{equation}
\end{lemma}

Taking $p=p_1$ in this lemma and then interpolating with (\ref{size111}), we get
\begin{equation}\label{interpolation1}
\|S_{1,T}(f)\|_{p_2'} \leq C|I_T|^{1/p'_2}{\rm{\bf size}}_1(T)^{p_1/p_2'} {\rm min}\{1, |F_1|/|F_3|\}^{1/p_1-1/p'_2}
\end{equation}

The second BMO estimate is localized to a single tree.
\begin{lemma}\label{BMO2}
Let $T\subset S$ be a tree, then 
\begin{equation}
\Big\|\Big(\sum\limits_{(j,n)\in T}\big|\Id^{**}_{n,l,j}\psi_{j_l+j+m}f\ast\Phi_{j_l+j+m}\big|^2\Big)^{1/2}\Big\|_{BMO} \leq C 2^{m} {\rm{\bf size}}_1(T).
\end{equation}
\end{lemma}
We provide the proof of Lemma \ref{BMO1} in section \ref{P_BMO1} and omit the proof Lemma \ref{BMO2} because it is a standard BMO-estimate, similar to (\ref{pBMO}). Interpolating the above lemma with (\ref{size111}), we obtain 
\begin{equation}
\bigg\|\bigg(\sum\limits_{(j,n)\in T}\big|\Id^{**}_{n,l,j}\psi_{j_l+j}f\ast\Phi_{j_l+j+m}\big|^2\bigg)^{1/2}\bigg\|_{p_2'} \leq C 2^{m(1-p_1/p'_2)} {\rm{\bf size}}_1(T)|I_T|^{1/p'_2}.
\end{equation}
Thus 
\begin{equation}\label{interpolation2}
\|S_{1,T}(f)   \|_{p'_2}\leq C2^{m(1-p_1/p'_2)} {\rm{\bf size}}_1(T)|I_T|^{1/p'_2}.
\end{equation}

Combining (\ref{interpolation1}) and (\ref{interpolation2}), we have:

\begin{lemma}\label{interpolation}
\begin{equation}\label{size*}
 \|S_{1,T}(f)\|_{p'_2}\leq
 C|I_T|^{1/p'_2}{\rm{\bf size}}_1^*(T),
\end{equation}
where
\begin{align*}
{\rm{\bf size}}_1^*(S) = \min
\Big\{2^{m(1-p_1/p'_2)} {\rm{\bf size}}_1(S),
{\rm{\bf size}}_1(S)^{p_1/p_2'} \min\{1, |F_1|/|F_3|\}^{1/p_1-1/p'_2}
\Big\}.
\end{align*}
\end{lemma}

\subsection{Completion of the proof}\q\q\q\q\q
\\

To finish the proof of Proposition \ref{major1}. We need the following lemmas. 

\begin{lemma}\label{remove00}
Let $k=1$ or $2$, $T\subset S_0$ a tree and $P\subset S_0$ a subset. Suppose that $T\cap P =\emptyset$ and $T$ is a maximal tree in $P\cup T$, then we have 
\begin{align}\label{remove}
|\Lambda_{P\cup T}(f,g)-\Lambda_{P}(f,g)-\Lambda_{T}(f,g)| \leq C \, {\rm{\bf size}}_1^*(P\cup T) {\rm{\bf size}}_2(P\cup T) |I_T|.
\end{align}
\end{lemma}
The proof  of Lemma \ref{remove00} is presented in Section \ref{P_remove00}. 
\begin{lemma}\label{organize}
Let $S\subset S_0$, for $k=1$ and $k=2$, $S$ can be partitioned into two parts $S_1$ and $S_2$ such that $S_1= \bigcup\limits_{T\in \mathcal F} T$ is a union of maximal trees with 
\begin{equation}\label{foruse0}
\bigcup\limits_{T\in \mathcal F} |I_T| \leq C |F_k|/{\rm{\bf size}}_k(S)^{p_k}
\end{equation} 
and 
\begin{equation}\label{foruse}
{\rm{\bf size}}_k(S_2) \leq \left(\frac{1}{2}\right)^{1/p_k} {\rm{\bf size}}_k (S), 
\end{equation}
where $C$ is the an absolute constant coming from the weak $(1, 1)$ norm of  the Hardy-Littlewood maximal function. 
\end{lemma}

\begin{proof}
This is direct results from (\ref{1size}) and (\ref{2size}). Choose a maximal tree $T\subset S$ with 
\begin{align}\label{choose_tree}
k\mbox{-}{\rm{\bf size}}(T) > \Big(\frac{1}{2}\Big)^{1/p_k} {\rm{\bf size}}_k (S). 
\end{align}
If ${\rm{\bf size}}_k(S\backslash T) > (\frac{1}{2})^{1/p_k} {\rm{\bf size}}_k (S)$, we can choose a maximal tree $T'\subset S\backslash T$ satisfying (\ref{choose_tree}) and then remove $T'$ from $S\backslash T$. Repeat this process we finally obtain:
\begin{align*}
S_1= \bigcup\limits_{T\in \mathcal F} T  
\end{align*}
such that (\ref{choose_tree}) is true for all $T\in \mathcal F$ and (\ref{foruse}) holds  for $S_2 = S\backslash S_1$. If remains to check (\ref{foruse0}) for $\{T\}_{T\in \mathcal F}$. Indeed, from the above procedure we see that the intervals $\{I_T\}_{T\in \mathcal F}$ are disjoint and by (\ref{1size}), (\ref{2size}), 
\begin{align}\label{unit-contain}
\bigcup\limits_{T\in \mathcal F}I_T\subseteq \bigg\{x: M_{p_k}(f_k)(x)\geq \left(\frac{1}{2}\right)^{1/p_k} {\rm{\bf size}}_k (S)\bigg\},
\end{align}  
where $f_1=f$ and $f_2=g$. Then (\ref{foruse0}) follows from the fact that Hardy-Littlewood maximal function
is bounded from $L^1$ to weak $L^1$. 
\end{proof}


%

%
%
%
%
%
%
%
%
%
%
%
%
%
%
%

Applying Lemma \ref{organize}, $S_0$ can be decomposed as follows:
\begin{equation}\label{S0-111}
S_0=\bigcup\limits_{\sigma\leq 1}S_\sigma,
\end{equation}
where $\sigma$ ranges over positive dyadic numbers and $S_\sigma$ is a union of maximal trees such that for each $T\in S_\sigma$, 
\begin{equation}\label{upper-111}
{\rm{\bf size}}_k(T)\leq \sigma^{1/p_k} (|F_1|/|F_3|)^{1/p_k}, \,\,\,\,\,\,\mbox{for} \,\,k=1\,\, \textbf{and}\,\, 2
\end{equation}
and 
\begin{equation}\label{lower-111}
{\rm{\bf size}}_k(T)\geq \left(\frac{1}{2}\sigma\right)^{1/p_k} (|F_1|/|F_3|)^{1/p_k}, \,\,\,\,\,\, \mbox{for} \,\, k=1 \,\, \textbf{or} \,\,2.  
\end{equation}
Now we turn to the proof of the (\ref{Lambdagoal}), which implies Proposition \ref{major1}. 
From (\ref{unit-contain}) and (\ref{lower-111}), we conclude that
\begin{equation}\label{.}
\sum\limits_{T\in S_\sigma}|I_T| \leq \frac{C|F_1|}{\left((\frac{1}{2}\sigma)^{1/p_1}(|F_1|/|F_3|)^{1/p_1}\right)^{p_1}}
 +
\frac{C|F_2|}{\left((\frac{1}{2}\sigma)^{1/p_2}(|F_2|/|F_3|)^{1/p_2}\right)^{p_2}} 
\leq C |F_3|/\sigma.
\end{equation}
In addition, for any single tree $T$,  using (\ref{***}), the definition of sizes and Lemma \ref{interpolation}, we obtain 
\begin{equation}\label{dot}
\Lambda_T(f, g) \leq C|I_T| {\rm \bf size}_1^*(T) {\rm \bf size}_2(T)\,. 
\end{equation}

Hence from Lemma \ref{remove00}, (\ref{S0-111}), (\ref{.}), and (\ref{dot}),  we get 
\begin{align}
\Lambda_{S_0}(f,g)
&\leq \sum\limits_{\sigma \leq 1} \sum\limits_{T\in S_\sigma}
{\rm{\bf size}}_1^*(S_\sigma){\rm{\bf size}}_2(S_\sigma) |I_T|,
\end{align}
which is dominated by 
\begin{align*}
&\q\q\q\q\q\q
\sum\limits_{\sigma \leq 1} \bigg(\sum\limits_{T\in S_\sigma}|I_T|\bigg)  \sigma^{1/p_2}\bigg(\frac{|F_2|}{|F_3|}\bigg)^{1/p_2}\,\, \cdot 
\\
&\min
\bigg\{
2^{m(1-p_1/p_2')}\sigma^{1/p_1}\bigg(\frac{|F_1|}{|F_3|}\bigg)^{1/p_1},
\bigg(\sigma^{1/p_1}\bigg(\frac{|F_1|}{|F_3|}\bigg)^{1/p_1}\bigg)^{p_1/p_2'} \!\!\!\!\!\min\Big\{1, \frac{|F_1|}{|F_3|}\Big\}^{\frac{1-p_1/p_2'}{p_1}}
\bigg\}
\cdot
\end{align*}
Here we used Lemma \ref{interpolation} and (\ref{upper-111}). 
From (\ref{.}), $\Lambda_{S_0}(f,g)$ is then controlled by
\begin{align*}
&\q\q\q\q\q\q C|F_3|\left(\frac{|F_1|}{|F_3|}\right)^{1/p_1}\left(\frac{|F_2|}{|F_3|}\right)^{1/p_2}\,\, \cdot 
\\
&\sum\limits_{\sigma\leq 1}\min\bigg\{2^{m(1-p_1/p_2')}\sigma^{1/p_1},
\sigma^{1/p_2'}\left(\frac{|F_1|}{|F_3|}\right)^{1/p_2'-1/p_1}\min\{1,\frac{|F_1|}{|F_3|}\}^{1/p_1-1/p_2'}
 \bigg\}\sigma^{1/p_2-1},
\end{align*}
which is also majorized by
\begin{align*}
|F_1|^{1/p_1}|F_2|^{1/p_2}|F_3|^{1/r'}\sum\limits_{\sigma\leq 1}\min\{2^{m(1-p_1/p_2')}\sigma^{1/p_1},
\sigma^{1/p_2'}
 \}\sigma^{1/p_2-1}\,.
\end{align*}
The desired inequality now follows by noticing 
\begin{align}
\sum\limits_{\sigma\leq 1}\min\{2^{m(1-p_1/p_2')}\sigma^{1/p_1},
\sigma^{1/p_2'}
 \}\sigma^{1/p_2-1} \leq Cm.
\end{align}

%
%
%
%
%
%
%
%
%
%
%
%
%
%
%
%

\section{Bilinear Maximal functions}\label{maxfunction}
\setcounter{equation}0

In this section, we study the following bilinear maximal function
\begin{align}\label{max}
M_{\Gamma_P}(f,g)(x)  = \sup_{\epsilon}\frac{1}{2\epsilon}\int_{-\epsilon}^{\epsilon} |f(x-t)g(x-P(t))|dt
\end{align}
and employ the known results in the previous sections to deduce that $M_{\Gamma_P}$ maps from $L^{p_1}\times L^{p_2}$ to $L^r$ uniformly, for the same range of $p_1,p_2$ and $r$.  

Observe that 
\begin{align*}
M_{\Gamma_P}(f,g)(x)  
\leq  
\sup_{\epsilon}\frac{1}{2\epsilon}  & \left(\int_{-\epsilon}^{-\epsilon/2} |f(x-t)g(x-P(t))|dt + 
 \int_{\epsilon/2}^{\epsilon} |f(x-t)g(x-P(t))|dt\right)
 \\
+ & \q\q\q\sup_{\epsilon}\frac{1}{2\epsilon}\int_{-\epsilon/2}^{\epsilon/2} |f(x-t)g(x-P(t))|dt
\end{align*}
and the last term in the right side is bounded by $\frac{1}{2}M_{\Gamma_P}(f,g)(x)$. Thus, it is sufficient to consider the following operator
\begin{align}\label{max1}
T^{**}(f,g)(x)  =
\sup_{\epsilon}\frac{1}{2\epsilon}\left(\int_{-\epsilon}^{-\epsilon/2} |f(x-t)g(x-P(t))|dt+
 \int_{\epsilon/2}^{\epsilon} |f(x-t)g(x-P(t))|dt\right).
\end{align}
Let $\rho$ be a nonnegative even Schwartz function supported on $(1/2,2)\cup (-2,-1/2)$ and $\rho(1)=1$, it is then enough for us to consider the  following smoothing version of $T^{**}$: 
\begin{align}\label{max2}
T^*(f,g)(x)  = \sup_{j\in \mathbb Z}\int f(x-t)g(x-P(t))\rho_j(t)dt = \sup_j T_j(f, g)(x),
\end{align}
where we assume $f$ and $g$ to be nonnegative. 

Like what we did in the case of bilinear Hilbert transforms, we divide $j$'s into $J_{\rm good}(N)$ and $J_l(N)$'s according to the polynomial $P$. Via a technical treatment as in section \ref{decomposition3}, we consider $j\in J^*_l(N)$. Applying the decomposition in section \ref{decomposition3}, we have 
\begin{align*}
 T_{j+j_l}(f,g)(x)  &=  \int \int \hat f(\xi) \hat g(\eta) e^{2\pi i (\xi+\eta)x} 
\sum\limits_{m\in\mathbb{Z}}\sum\limits_{n\in\mathbb{Z}}\mathcal{M}_{m,n}(\xi,\eta)
d\xi d\eta
\\
&\!\!:=\sum_{m,n}T_{j_l+j}^{(m,n)}(f,g)(x)
 \end{align*}
where
\begin{align}\label{kernel}
\mathcal{M}_{m,n}(\xi,\eta)=\hat\Phi(\frac{\xi}{2^{j_l+j+m}})\hat\Phi(\frac{\eta}{2^{j_l+lj+n}})\left( \int e^{-2\pi i (\frac{t\xi}{2^{j_l+lj}} +\frac{(t^l+Q_l(t))\eta}{2^{j_l+lj}})}\rho(t)dt \right)
\end{align}
and $j_l$ is defined as in section \ref{decomposition3}.
The next step is to divide $(m,n)$ into the following cases:\\
$\bullet$ {Case 1.}  $|m-n|>>1$ and $\min\{m,n\}>0$,\\
$\bullet$ {Case 2.}  $|m-n|>>1$ with  $\max\{m,n\}>0, \min\{m,n\}<0$,\\
$\bullet$ {Case 3.}  $\max\{m,n\}<0$,\\
$\bullet$ {Case 4.}  $m\sim n$ and $ m, n>0$.\\
We want to point out here that case 4 is essential, however it can be handled similarly as bilinear Hilbert transforms. We begin to analyze the first case. Without lost of generality, we will assume $m>n>0$. Applying Fourier Series to write (\ref{kernel})  as
\begin{align}\label{series}
\mathcal{M}_{m,n}(\xi,\eta)=\hat\Phi(\frac{\xi}{2^{j_l+j+m}})\hat\Phi(\frac{\eta}{2^{j_l+lj+n}})\sum_{(n_1,n_2)\in\mathbb Z\times \mathbb Z}C^{(m,n)}_{n_1,n_2}e^{2\pi i \frac{n_1\xi}{2^{j_l+j+m}}}
e^{2\pi i \frac{n_2\eta}{2^{j_l+lj+n}}}.
\end{align} 
By utilizing integration by parts, it is clear that the coefficients $C^{(m,n)}_{n_1,n_2}$   decay rapidly  in the sense that for every positive integer $K$
\begin{align}\label{decay}
|C^{(m,n)}_{n_1,n_2}| \leq C_K (1+n_1^2+n_2^2)^{-K}(1+2^m+2^n)^{-K}.
\end{align}
Notice 
\begin{align*}
\sup_{j\in J^*_l(N)}\left|\int \hat f(\xi)2^{2\pi i x\xi}\hat\Phi(\frac{\xi}{2^{j_l+j+m}})
e^{2\pi i \frac{n_1\xi}{2^{j_l+j+m}}}d\xi\right| \leq C(1+n_1^2)Mf(x)
\end{align*}
and
\begin{align*}
\sup_{j\in J^*_l(N)}\left|\int \hat g(\eta)2^{2\pi i x\eta}\hat\Phi(\frac{\eta}{2^{j_l+lj+m}})
e^{2\pi i \frac{n_2\eta}{2^{j_l+lj+m}}}d\eta\right| \leq C(1+n_2^2)Mg(x).
\end{align*}
Here we use the fact that for any Schwartz function $\Phi$, $f\ast \Phi_j(x-n/2^j) \leq C_{\Phi}(1+n^2)Mf(x)$. Henceforth we have 
\begin{align}
\sup_{j\in J^*_l(N)}|T^{(m,n)}_{j_l+j}(f,g)(x)| \leq \sum_{n_1,n_2}|C_{n_1,n_2}^{(m,n)}|(1+n_1^2)(1+n_2^2)Mf(x)Mg(x)\,.
\end{align}
Summing up all $m, n$'s under the assumption of Case 1 and using (\ref{decay}), we get the upper bound 
\begin{align}
 C Mf(x) Mg(x)\,. 
\end{align}
Case 2 is a combination of Case 1 and Case 3, thus can be handled by methods applied in Case 1 and Case 3, we omit the details. We now turn to Case 3. Consider
\begin{align}
T^-_{j_l+j}(x)= \sum_{m<0,n<0}T_{j_l+j}^{(m,n)}(f,g)(x)
\end{align}
where the corresponding kernel is 
\begin{align*}
\mathcal{M}^-(\xi,\eta) &=\sum_{m<0,n<0}\mathcal{M}_{m,n}(\xi,\eta)
\\
&=\sum_{m<0,n<0}\hat\Phi(\frac{\xi}{2^{j_l+j+m}})\hat\Phi(\frac{\eta}{2^{j_l+lj+n}})\left( \int e^{-2\pi i (\frac{t\xi}{2^{j_l+lj}} +\frac{(t^l+Q_l(t))\eta}{2^{j_l+lj}})}\rho(t)dt \right)
\\
&= \hat\Psi(\frac{\xi}{2^{j_l+j}})\hat\Psi(\frac{\eta}{2^{j_l+lj}})\left( \int e^{-2\pi i (\frac{t\xi}{2^{j_l+lj}} +\frac{(t^l+Q_l(t))\eta}{2^{j_l+lj}})}\rho(t)dt \right),
\end{align*}
where 
$$
\hat\Psi(\frac{\xi}{2^{j_l+j}})= \sum_{m<0}\hat\Phi(\frac{\xi}{2^{j_l+j+m}})
$$
and
$$
\hat\Psi(\frac{\eta}{2^{j_l+lj}})= \sum_{n<0}\hat\Phi(\frac{\eta}{2^{j_l+lj+n}}).
$$
Invoking Taylor expansion, we expand $\mathcal{ M}^-(\xi,\eta)$ as
$$
 \sum_{k_1\geq 0}\sum_{k_2 \geq 0}\frac{1}{k_1!}\frac{1}{k_2!}\hat\Psi(\frac{\xi}{2^{j_l+j}})(\frac{2\pi i \xi}{2^{j_l+j}})^{k_1}\hat\Psi(\frac{\eta}{2^{j_l+lj}})(\frac{2\pi i\eta}{2^{j_l+lj}})^{k_2}\int \rho(t)
t^{k_1} (t^l+Q_l(t))^{k_2}dt.
$$
This yields 
\begin{align*}
\sup_{j\in J^*_l(N)}|T^-_{j_l+j}(x)| \leq CMf(x)Mg(x).
\end{align*}
Therefore for $(m,n)$ in Case 1, 2 or 3, we have the following point wise estimate
\begin{align}\label{nond}
\sup_{j\in J^*_l(N)}\sum T_{j_l+j}^{(m,n)}(f,g)(x) \leq CMf(x)Mg(x).
\end{align}
This operator is certainly bounded from $L^{p_1}\times L^{p_2} \to L^{r}$, for $p_1 > 1$, $p_2>1$, $r>1/2$ and $\frac{1}{p_1}+\frac{1}{p_2} =\frac{1}{r}$. For Case 4, notice that if $m$ is fixed,the number of $n$ is finite. WLOG, we can assume $m=n$. Notice that 
\begin{align*}
\sup_{j\in J^*_l(N)}T_{j+j_l}^{(m,m)}(f,g)(x) \leq \sum_{j\in J_l^*{(N)}}|T_{j_l+j}^{(m,m)}(f,g)(x)|\,.
\end{align*}
  We conclude that $\sum\limits_{j\in J_l^*(N)}|T_{j_l+j}^{(m,m)}(f,g)(x)|$ is bounded from $L^{p_1}\times L^{p_2} \to L^{r}$ with a bound $2^{-\epsilon m}$, for $p_1 > 1$, $p_2>1$, $r>1/2$ and $\frac{1}{p_1}+\frac{1}{p_2} =\frac{1}{r}$. 
  This is a simple consequence of Proposition \ref{progoal} and (\ref{single}).
Hence we completes  the case when ${j}\in J_l^*(N)$. Now we turn to $J_{\rm good}(N)$. Similarly, 
\begin{align*}
\sup_{j\in J_{\rm good}(N)}|T_j(f,g)(x)| 
\leq
\sum_{j\in J_{\rm good}(N)} |T_{j}(f,g)(x)|.
\end{align*}
Since $\# J_{\rm good}(N)$ does not exceed a uniform constant  by Lemma \ref{keyL} and each $T_j$ is bounded for $p_1 > 1$, $p_2>1$, $r>\frac{d-1}{d}$ by Theorem \ref{badL}, we obtain the desired estimate and  finish the proof of Theorem \ref{maxbi} for the bilinear 
maximal functions.

%
%
%
%
%
%
%
%
%
%
%
%

\section{Proof of Lemma \ref{error}}\label{P_error}
\setcounter{equation}0

We prove Lemma \ref{error} in this section. The main difficulty here is to obtain an upper bound 
growing slowly, say $O(m)$. To achieve this, we utilize a Whitney decomposition for the exceptional 
set $\Omega$ defined as in (\ref{defofExp}). More precisely,
let $\mathcal F$ be a collection of pairwise disjoint dyadic intervals $J's$ such that for each $J\in\mathcal F$
$$
  3|J| \geq  {\rm dist}(J, \p\Omega)\geq |J|\,
$$
and $\Omega=\bigcup_{J\in\mathcal F} J$. 
From the defintion of $\Omega$, clearly we have, 
for each $J\in\mathcal F$ and $i\in\{1,2\}$, 
$$
\frac{1}{|7J|}\int_{7J}\Id_{F_i}\leq C|F_i|/|F_3|\,.  
$$

To simplify the notations, we introduce following functions. 
For any integer $M$, let $\delta_{j, M}: \mathbb R\times \mathbb R\to \mathbb R$ be 
given by
\begin{equation}
\delta_{j, M}(x, y) = \big( 1+ 2^{j_l+j+m}|x-y| \big)^{-M}\,\, {\rm for}\,\, (x, y)\in\mathbb R\times \mathbb R\,.  
\end{equation}
For any $x\in\mathbb R$ and any measurable set $E\subset \mathbb R$, 
\begin{equation}
\delta_j(x, E) := \big( 1+ 2^{j_l+j+m}{\rm dist}(x, E)\big)^{-N}\,.
\end{equation}
Also define $x_j(t)$ and $x_{j, l}(t)$ as
\begin{equation}
x_j(t)=x-2^{-j-j_l}t\,\,\, {\rm and}\,\,\, x_{j, l}(t)=x-2^{-j_l-lj}(t^l+Q_l(t))\,. 
\end{equation}
and set
\begin{equation}
 j'= j_l+j+m\,\,{\rm and}\,\,\,  j''=j_l+lj+m\,.
\end{equation}

\begin{lemma}\label{lem0}
Let $I_1, I_2$ be two intervals in $\mathcal F$. Suppose $a\in I_1$ and $b\in I_2$. 
If ${\rm dist}(I_1, I_2)\geq 100\min\{|I_1|, |I_2|\}$,
then    
$$
\delta_{j, 2N}(a,b)
\leq  C_N \big(\delta_{j}(a, \Omega^c) +\delta_j(b, \Omega^c) \big)\,. 
$$
\end{lemma}

\begin{proof}
 In fact, we prove a stronger result:
\begin{equation}
|a-b|\geq C{\rm dist(a, \Omega^c)}\,. 
\end{equation}
It is easy to see that Lemma \ref{lem0} follows immediately by the triangle inequality. 
First we consider the case when  $|I_1|\leq |I_2|$. In this case,  because $I_1$ is in 
the Whitney decomposition $\mathcal F$,  there exists $z_0\in\Omega^c$ such that
$
|a-z_0|\leq 3|I_1|\,. 
$
Thus 
$$ 
 |a-b|\geq 100|I_1|\geq \frac{100}{3}|a-z_0|\geq {\rm dist}(a, \Omega^c)\,,
$$
as desired.  We now turn to the second case when $|I_1| > |I_2|$. 
 In this case,  
notice that there exists $z_0\in\Omega^c$ such that
$
|b-z_0|\leq 3|I_2|$, since $I_2\in\mathcal F$. 
Thus 
$$
100|I_2|\leq |a-b|\leq |a-z_0|+|b-z_0|\leq |a-z_0|+3|I_2|\,,
$$
which implies
$$
|a-z_0|\geq 97|I_2|\geq \frac{97}{3}|b-z_0|\,.
$$
From this, we have
$$
 |a-b|\geq |a-z_0|-|b-z_0|\geq |a-z_0|-\frac{3}{98}|a-z_0|
\geq \frac{95}{98}{\rm dist}(a, \Omega^c)\,. 
$$
\end{proof}

\begin{lemma}\label{lem00}
Let $I_1, I_2$ be two intervals in $\mathcal F$. Suppose  ${\rm dist}(I_1, I_2)\leq 100\min\{|I_1|, |I_2|\}$,
then    
$$
|I_1|\sim |I_2|\,. 
$$
\end{lemma}

\begin{proof}
Assume $|I_1|\leq |I_2|$ and we prove that $|I_1|\geq C|I_2|$. 
In addition, we can assume that $|I_1|< \frac{1}{2000}|I_2|$, otherwise just simply take
$C=1/2000$.  From  ${\rm dist}(I_1, I_2)\leq 100\min\{|I_1|, |I_2|\}$, we have 
$${\rm dist}(I_1, I_2)\leq \frac{1}{20}|I_2|\,. $$
This indicates $I_1$ is in a small neighborhood of $I_2$. Since $I_2\in\mathcal F$, 
${\rm dist}(I_2, \Omega^c)\geq |I_2|$. However, from ${\rm dist}(I_1, \Omega^c)\leq 3|I_1|$, 
we conclude that there exists $z_0\in\Omega^c$ such that $ {\rm dist}(I_1, z_0)\leq 3|I_1|$. 
Thus ${\rm dist}(z_0, I_2)\leq 3|I_1|+\frac{1}{20}|I_2| < \frac{1}{10}|I_2|$, which implies
${\rm dist}(I_2, \Omega^c)< \frac{1}{10}|I_2|$, contradicting to ${\rm dist}(I_2, \Omega^c)\geq |I_2|$. 
\end{proof}

Now we are ready to prove Lemma \ref{error}. We only present a proof of  (\ref{error1}) for $L_{1, m}(f, g)$. 
(\ref{error2}) can be done similarly.
Inserting the absolute value in the integrand, we bound $L_{1, m}(f, g)$ by $L'_m(f, g)$, where
$L'_{m}(f, g)$  is given by
$$
\sum_{j}\int_{\Omega^c}
 \int_{\Omega_{j'}} 2^{j'} \delta_{j, N}(x_j(t)-y) dy 
\int \left|f*\Phi_{ j'} (x_j(t))
 g*\Phi_{ j''}(x_{j,l}(t))\rho(t)\right| dt  dx.
$$

If we restrict the variable $x_j(t)$ to $\Omega^c$, then the situation is easy. Indeed, 
in this case, $L'_m$ is controlled by, from the definition of $\Omega$, 
\begin{align*}
\sum_{j}\int_{\Omega^c}\int_{\Omega_{j'}}
2^{j'} \Id_{\Omega^c}(x_j(t))
\delta_{j, N}\big( x_j(t)-y\big) dy
\left(\frac{|F_1|}{|F_3|}\right)^{1/p_1}\!
\int \left| g*\Phi_{j''}(x_{j,l}(t))\rho(t)\right| dt  dx.
\end{align*}

Changing variable $x_j(t)\mapsto u$, we dominate $L_m'$ by 
$$
  \sum_{j}\int_{\Omega_{j'}} \left(\frac{|F_1|}{|F_3|}\right)^{1/p_1}
\delta_{j}\big( y, \Omega^c \big)
\big(\int_{\Omega^c} \int 2^{j'} \delta_{j,N}(u,y) M\Id_{F_2}(u+\textbf{tr}_j(t)) |\rho(t)| du dt \big)   dy.
$$
For $|j|\geq N$ and any fixed $u$, a change of variable $u+\textbf{tr}_j(t)\mapsto w$ yields
an estimate of $L'_m$ by the following two terms, 
\begin{eqnarray*}
 & & \sum_{j>0}\int_{\Omega_{j'}}\left(\frac{|F_1|}{|F_3|}\right)^{1/p_1} 
\delta_{j}(y, \Omega^c)
\int_{\Omega^c}2^{j'} \delta_{j, N}\big(u, y\big)  \\
& & \q\q\q
2^{j_l+j}\int_{u-C2^{-j_l-j}}^{u+C2^{-j_l-j}} \big( M\Id_{F_2}(w) \big)^{1/p_2}dw     du dy
\end{eqnarray*}
and
\begin{eqnarray*}
 & & \sum_{j<0}\int_{\Omega_{j'}}\left(\frac{|F_1|}{|F_3|}\right)^{1/p_1} 
\delta_{j}(y, \Omega^c)
\int_{\Omega^c}2^{j'} \delta_{j, N}\big(u, y\big)  \\
& & \q\q\q
2^{j_l+lj}\int_{u-C2^{-j_l-lj}}^{u+C2^{-j_l-lj}} \big( M\Id_{F_2}(w) \big)^{1/p_2}dw     du dy
\end{eqnarray*}


Using the fact that $(Mg)^{1/p_2}$ is an $A_1$ weight, we then have an upper bound 
$$
C\left(\frac{|F_1|}{|F_3|}\right)^{1/p_1}\left(\frac{|F_2|}{|F_3|}\right)^{1/p_2} |\Omega|
$$
as desired. Henceforth we may restrict $x_j(t)$ to $\Omega$.  
In this case,  apply the Whitney decomposition for $\Omega$ to estimate $L_m' (f, g)$ by
\begin{eqnarray*}
 & &\sum_j\sum_{I_1, I_2\in\mathcal F}
\int_{\Omega^c}  \int \left( \int_{\Omega_{j'}}
2^{j'}\Id_{I_1}(y)\Id_{I_2}\big(x_j(t)\big) \delta_{j, 2N}\big(x_j(t)-y\big) dy \right)\\
& &\q\q\q\q
\bigg|f*\Phi_{j'} (x_j(t))
 g*\Phi_{j''}\big(x_{j, l}(t)\big)\rho(t)\bigg| dt  dx\,. 
\end{eqnarray*}


If ${\rm dist}(I_1, I_2)\geq 100\min\{|I_1|, |I_2|\}$, then apply Lemma \ref{lem0} to 
gain the decay factors 
$\delta_j(y, \Omega^c)\delta_j(x_j(t), \Omega^c) $
in the integrand. These decay factors allow us to treat this case exactly as the easy case when $x_j(t)\in \Omega^c$. We omit the details. \\

The argument above  indicates that when$x_j(t)\in\Omega$ the principal contribution 
is made by those $I_1, I_2$ with short distance. 
Thus we have to face this main difficulty by estimating
\begin{eqnarray*}
L''_m & =  & \sum_{j}\!\!\!\!\!\sum_{\substack{I_1, I_2\in\mathcal F\\
 {\rm dist}(I_1, I_2)< 100\min\{|I_1|, |I_2|\}   }} \!\!
\int_{\Omega^c}\int \!\!\!\int_{\Omega_{j'}}     2^{j'}\Id_{I_1}(y)\Id_{I_2}\big((x_j(t)\big)\\
& &\q\q\cdot
\left|f*\Phi_{j'}\big (x_j(t)\big)
 g*\Phi_{j''}(x_{j, l}(t))\rho(t)\right| \delta_{j, N}\big(x_j(t)-y\big)
dy  dt  dx
\end{eqnarray*}

From Lemma {\ref{lem00}}, we know that $|I_1|\sim |I_2|$ under the assumption 
${\rm dist}(I_1, I_2)< 100\min\{|I_1|, |I_2|\}$. Moreover, we are free to assume that
\begin{equation}\label{eqnj}
\max\{|I_1|, |I_2|\}\leq 2^{50} 2^{-j_l-j}\,.
\end{equation}
 Indeed, if not so, then $|I_2|\geq 2^82^{-j_l-j}$
from Lemma \ref{lem00}. For $x_j(t)\in I_2 $ and $x\in\Omega^c$,
we clearly have 
$$
 {\rm dist}(x, I_2)\leq | x_j(t)-x|\leq 2\cdot 2^{-j_l-j}\leq \frac{1}{2^7}|I_2|\,, 
$$ 
 implying ${\rm dist}(I_2, \Omega^c)\leq |I_2|/10$. This contradicts $I_2\in\mathcal F$. 
On the other hand, we can also assume 
\begin{equation}\label{I1small}
 |I_1|\geq \frac{1}{2^{10}}2^{-j_l-j-m}\,. 
\end{equation}
This is because otherwise, we have $|I_1|< 2^{-j_l-j-m-10}$.
But notice  that $y\in I_1$ and $y\in\Omega_{j_l+j+m}$ imply 
$$
  2^{-j_l-j-m}\leq  {\rm dist}(y, \Omega^c) \leq 4|I_1| < 2^{-j_l-j-m-8}\,. 
$$
This yields a contradiction. From (\ref{eqnj}) and (\ref{I1small}), 
we conclude that for any given $I_1\in \mathcal F$, 
\begin{equation}\label{range} 
            \frac{2^{50}}{|I_1|}\geq 2^{j_l+j} \geq \frac{1}{2^{m+10}|I_1|}\,.
\end{equation}
(\ref{range}) states that for given $I_1$, there are at most $100m$ many
$j$'s.  From Lemma \ref{lem00}, we also see that 
for given $I_1$,  there are at most finitely many $I_2$'s with 
${\rm dist}(I_1, I_2)< 100\min\{|I_1|, |I_2|\} $.   Without loss of generality, 
in what follows, we assume $I_2=I_1$ and we try to estimate
\begin{align*}
L'''_m = &\sum_{I_1\in \mathcal F}\!\!\sum_{\substack{j \\
 \frac{2^{50}}{|I_1|}\geq 2^{j_l+j}\geq \frac{1}{2^{m+10}|I_1|}  }}
\int\!\int \!\int_{\Omega_{j'}} 2^{j'} \Id_{I_1}(y)\Id_{I_1}(u) \Id_{\Omega^c}(u+2^{-j_l-j}t)
\\
&\q\q\q\q \left|f*\Phi_{j'} (u)
      g*\Phi_{j''}(u+\textbf{tr}_j(t))\rho(t)\right| 
      \delta_{j, N}(u, y) dy
  dt  du
\end{align*}

Here we made a change variable $x_j(t)\mapsto u$. 
Notice that 
$$
\left|f*\Phi_{j'} (u)\right|  \delta_{j, 2}(u, y)
\leq CMf(y)
$$
$L'''_m$ is majorized by
\begin{align*}
 \sum_{I_1\in \mathcal F}\!\!\sum_{\substack{j\\
 \frac{2^{50}}{|I_1|}\geq 2^{j_l+j}\geq \frac{1}{2^{m+10}|I_1|}  }}\!\!
\int_{I_1} Mf(y)  \mathcal G(y) dy
\end{align*}
where $\mathcal G(y) $ is defined as 
\begin{align*}
\mathcal G(y)=
\!\!\iint  
2^{j'} \Id_{I_1}(u)\left|g*\Phi_{j''}(u+\textbf{tr}_j(t))\rho(t) \Id_{\Omega^c}(u+2^{-j_l-j}t)\right|
\delta_{j, N-2}(u-y) dt du \,.
\end{align*}

Using the same $A_1$-trick as before, we replace $u+\textbf{tr}_j(t)$ by $w$ for given $u$. 
Since $u+2^{-j_l-j}t\in\Omega^c$, we integrate in $w$ in a $2^{10}2^{-j_l-j}$ neighborhood of $u$ if $j>0$ or in a $2^{10}2^{-j_l-lj}$ neighborhood of $u$ if $j<0$.  The neighborhoods  must contain
at least one point from $\Omega^c$ in order to make a nonzero value for the integral.  Henceforth, we estimate
$L'''_m$ by
$$
 C\left(\frac{|F_2|}{|F_3|}\right)^{1/p_2}
 \sum_{I_1\in \mathcal F}\!\!\!\sum_{\substack{j \\
 \frac{2^{50}}{|I_1|}\geq 2^{j_l+j}\geq \frac{1}{2^{m+10}|I_1|}  }}\!\!\!\!\!\!
\int_{I_1} Mf(y) dy\,,
$$
which can be controlled by
\begin{align*}
&Cm\left(\frac{|F_2|}{|F_3|}\right)^{1/p_2}
 \sum_{I_1\in \mathcal F}7|I_1|
\frac{1}{|7I_1|}\int_{7I_1} \left( M\Id_{F_1}(y)\right)^{1/p_1} dy
\\
\leq &Cm\left(\frac{|F_2|}{|F_3|}\right)^{1/p_2}\sum_{I_1\in \mathcal F}|I_1|
\inf_{y\in 7I_1} \left( M\Id_{F_1}(y)\right)^{1/p_1}\,. 
\end{align*}
Since $I_1\in\mathcal F$,  it is clear that 
$$
\inf_{y\in 7I_1} M\Id_{F_1}(y) \leq C|F_1|/|F_3|\,. 
$$
Hence, $L'''_m$ can finally dominated by
$$
 Cm \left(\frac{|F_1|}{|F_3|}\right)^{1/p_1} \left(\frac{|F_2|}{|F_3|}\right)^{1/p_2}
 |\Omega|\,,
$$
as desired.

%
%
%
%
%
%
%
%
%
%
%
%
%
%
%
%
%
%
%
%
%
%

\section{Proof of Lemma \ref{BMO1}}\label{P_BMO1}
\setcounter{equation}0

In this section we provide a BMO estimate for the square function $S_{1, Z}$. 
Notice that 
\begin{equation}\big|\sum\limits_{n\in Z_j } \Id_{n,l,j}^*\psi_{j_l+j+m}\big| \leq C\,.
\end{equation}
 (\ref{lp}) is a direct consequence  of   the definition of $S_{1, Z}$ and  Littlewood-Paley $L^p$ theory 
 on the square function. 
\\

For the BMO estimate (\ref{pBMO}), let $\textbf{J}$ be a dyadic interval of length $2^{-J}$. 
From the definition of BMO norm,  we need to majorize the following value
\begin{align}\label{B1}
\inf_{c} \int_\textbf{J} \bigg|\bigg(\sum\limits_{j\in J^*_{l,+}(N)}\Big|\sum\limits_{n\in Z_j }f_{n,l,m,j}(x)\Big|^2\bigg)^{1/2} - c\bigg| dx.
\end{align}
Via the triangle inequality,  we dominate (\ref{B1}) by $C(I_1+I_2)$, where 
$$I_1=
 \int_\textbf{J} \bigg|\bigg(\sum\limits_{j\in J^*_{l,+}(N), J\leq j+l_j+m}\big|\sum\limits_{n\in Z_j }f_{n,l,m,j}(x)\big|^2\bigg)^{1/2}\bigg| dx
$$
and
$$I_2=  
\inf_{c\in\mathbb{R}}\int_\textbf{J} \bigg|\bigg(\sum\limits_{j\in J^*_{l,+}(N), J\geq j+l_j+m}\Big|\sum\limits_{n\in Z_j }f_{n,l,m,j}(x)\Big|^2\bigg)^{1/2} - c\bigg| dx.
$$
We consider $I_1$ first. By splitting $f_{n,l,m,j}(x)$ into two parts
$$
 \Id_{n,l,j}^*\psi_{j_l+j+m}\big(f\Id_{2\textbf{J}}\big)\ast\Phi_{j_l+j+m}(x)+\Id_{n,l,j}^*\psi_{j_l+j+m}\big(f\Id_{(2\textbf{J})^c}\big)\ast\Phi_{j_l+j+m}(x), 
$$
$I_1$ can be bounded by $I_{1,1}+I_{1,2}$, where 
 $$I_{1,1}=
 \int_\textbf{J} \bigg|\bigg(\sum\limits_{j\in J^*_{l,+}(N), J\leq j+l_j+m}\Big|\sum\limits_{n\in Z_j }\Id_{n,l,j}^*\psi_{j_l+j+m}(f\Id_{2\textbf{J}})\ast\Phi_{j_l+j+m}(x)\Big|^2\bigg)^{1/2}\bigg| dx
$$
and 
$$
I_{1,2}=
\int_\textbf{J} \bigg|\bigg(\sum\limits_{j\in J^*_{l,+}(N), J\leq j+l_j+m}\Big|\sum\limits_{n\in Z_j }\Id_{n,l,j}^*\psi_{j_l+j+m}(f\Id_{(2\textbf{J})^c})\ast\Phi_{j_l+j+m}(x)\Big|^2\bigg)^{1/2}\bigg| dx.
$$
Applying H\"older's Inequality and (\ref{lp}), we estimate $I_{1,1}$ by 
\begin{align*}
 &  |\textbf{J}|^{1/p'}
\bigg\|\bigg(\sum\limits_{ \substack{j\in J^*_{l,+}(N)\\J\leq lj+l_j+m} }\Big|\sum\limits_{n\in Z_j }\Id_{n,l,j}^*\psi_{j_l+j+m}(f\Id_{2\textbf{J}})\ast\Phi_{j_l+j+m}(x)\Big|^2\bigg)^{1/2} \bigg\|_{L^p({\textbf{J}})}
\\
\leq& C
|\textbf{J}|^{1/p'}\big\|f\Id_{2\textbf{J}}\big\|_{p} \big\|\psi_{j_l+j+m}\big\|_{L^{\infty}(\textbf{J})}
\\
\leq &
C|\textbf{J}| \min\{1, |F_1|/|F_3|\}^{1/p}
\end{align*}
We now turn to $I_{1,2}$. For each $x$, choose $z\in \Omega$ with ${\rm dist}(x,\Omega) \sim |x-z|$, then  
\begin{align*}
&\Big|(f\Id_{(2\textbf{J})^c})\ast \Phi_{j_l+j+m}(x)\Big|
\\
\leq &C_K \int_{(2\textbf{J})^c}|f(y)| 
2^{j_l+j+m} \delta_{j, K}(x, y) dy
\\
\leq &
C_K\int_{(2\textbf{J})^c}
2^{j_l+j+m}|f(y)| \delta_{j,2}(y, z)\frac{1}{\delta_{j, 2}(y,x)\delta_{j,2}(x,z)} 
\delta_{j, K}(x, y) dy
\end{align*}
From $x\in\textbf{J}$ and $y\in (2\textbf{J})^c$, we have $|x-y| \geq 2^{-J-1}$, which implies
\begin{align}\label{ui1}
|(f\Id_{(2\textbf{J})^c})\ast \Phi_{j_l+j+m}(x)|
\leq CMf(z)(1+2^{j_l+j+m-J-1})^{2-K} (1+2^{j_l+j+m}\rm dist(x,\Omega))^2.
\end{align}
By the definition of $\Phi_{j_l+j+m}$, we have 
\begin{align}\label{ui2}
|\Phi_{j_l+j+m}(x)(1+2^{j_l+j+m}\rm dist(x,\Omega))^2| \leq C.
\end{align}
Also, from the definition of $\Omega$,  $z\in\Omega$ implies 
\begin{align}\label{ui4}
Mf(z) \leq \min\{1, |F_1|/|F_3|\}\leq  \min\{1, |F_1|/|F_3|\}^{1/p}
.
\end{align}
By utilizing  $|\sum\limits_{n\in Z_j}\Id_{j_l+j+n}^*(x)| \leq 1$, (\ref{ui1}), (\ref{ui2}), 
and (\ref{ui4}), we obtain
\begin{align*}
&\bigg|\bigg(\sum\limits_{j\in J^*_{l,+}(N), J\leq j+l_j+m}\Big|\sum\limits_{n\in Z_j }\Id_{n,l,j}^*\psi_{j_l+j+m}(f\Id_{(2\textbf{J})^c})\ast\Phi_{j_l+j+m}(x)\Big|^2\bigg)^{1/2}\bigg|
\\
\leq &C \min\{1,|F_1|/|F_3|\}^{1/p}
\end{align*}
and henceforth $I_{1,2}$ is bounded by $C|\textbf{J}|\min\{1,|F_1|/|F_3|\}^{1/p}$. 

Finally we estimate $I_2$. Notice that 
 \begin{align*}
 I_2
&\leq
|\textbf{J}|^{1/2}
\bigg(\inf_{c}\int_\textbf{J} \bigg|\Big(\sum\limits_{j\in J^*_{l,+}(N), J\geq j+l_j+m}\big(\sum\limits_{n\in Z_j }f_{n,l,m,j}(x)\big)^2-c\Big) \bigg| dx\bigg)^{1/2}
\\
&\leq
|\textbf{J}|^{1/2}
\bigg(|\textbf{J}|\int_\textbf{J} \bigg|D\Big(\sum\limits_{j\in J^*_{l,+}(N), J\geq j+l_j+m}\big(\sum\limits_{n\in Z_j }f_{n,l,m,j}(x)\big)^2\Big) \bigg | dx\bigg)^{1/2}.
\end{align*}
For the last inequality, we applied the Poincar\'e Inequality. 
To get the desired estimate for $I_2$,  it suffices to prove the following inequality 

\begin{align}\label{toprove}
\int_\textbf{J} \bigg|D\bigg(\sum\limits_{\substack{j\in J^*_{l,+}(N)\\J\geq j+l_j+m}}\Big(\sum\limits_{n\in Z_j }f_{n,l,m,j}(x)\Big)^2\bigg) \bigg| dx \leq C \min\{1, |F_1|/|F_3|\}^{2/p}.
\end{align}
By the chain rule for the differential operator $D$, we have 
\begin{align}\label{tem}
\bigg |D\Big(\sum\limits_{n\in Z_j }f_{n,l,m,j}(x)\Big)^2 \bigg |
=2 \bigg|\sum\limits_{n\in Z_j }f_{n,l,m,j}(x)\bigg|
\bigg|\sum\limits_{n\in Z_j }Df_{n,l,m,j}(x)\bigg|
\end{align}
Using a similar argument in the estimation of $I_{1,2}$,  we get
$$
|\sum\limits_{n\in Z_j }f_{n,l,m,j}(x)| \leq C\min\{1, |F_1|/|F_3|\}.
$$
Moreover, the product rule yields 
\begin{align*}
 \bigg|\sum\limits_{n\in Z_j }Df_{n,l,m,j}(x) \bigg| \leq \bigg|D\Big(\sum\limits_{n\in Z_j }\Id^*_{j_l+j+m}(x)\Big)\Phi_{j_l+j+m}(x)f\ast \Phi_{j_l+j+m}(x)\bigg |
 \\ +
  \bigg|\Big(\sum\limits_{n\in Z_j }\Id^*_{j_l+j+m}(x)\Big)\Big(D\Phi_{j_l+j+m}(x)\Big)f\ast \Phi_{j_l+j+m}(x)\bigg|
\\+
\bigg |\Big(\sum\limits_{n\in Z_j }\Id^*_{j_l+j+m}(x)\Big)\Phi_{j_l+j+m}(x)f\ast \Big(D\Phi_{j_l+j+m}(x)\Big)\bigg|
\end{align*}
Due to the differential operator, each of the single term in the previous sum is dominated by $C2^{j_l+j+m} \min\{1, |F_1|/|F_3|\}$. Therefore 
\begin{align*}
&\int_\textbf{J} \bigg|D\bigg(\sum\limits_{j\in J^*_{l,+}(N), J\geq j+l_j+m}\Big(\sum\limits_{n\in Z_j }f_{n,l,m,j}(x)\Big)^2\bigg) \bigg| dx 
\\
\leq  &C\min\{1, |F_1|/|F_3|\}^2\int_{\textbf{J}}\sum\limits_{J\geq j_l+j+m}2^{j_l+j+m}dx
\leq C \min\{1, |F_1|/|F_3|\}^{2/p},
\end{align*}
Hence we complete the proof of (\ref{toprove}) and thus Lemma \ref{BMO1}.

%
%
%
%
%
%
%
%
%
%
%
%
%
%
%
%


%
%
%
%
%
%
%

\

%
%
%
%
%
%
%
%
%
%
%
%
%
%
%
%
%
\section{proof of Lemma \ref{remove00}}\label{P_remove00}
\setcounter{equation}0

To prove the lemma, we introduce a function $d_j: S_0\times S_0\mapsto \mathbb R$ by
\begin{equation}
 d_j(S_1, S_2) = \Big(1+2^{j_l+j+m}{\rm dist}(S_1, S_2)\Big)^{-K}
\end{equation}
for $S_1, S_2\subset S_0$. Here $K$ is a sufficiently large constant.

The following lemma is needed in our proof of Lemma {\ref{remove00}}. 

\begin{lemma}\label{claim1}
Let $T\subset S_0$ be a tree, and $P$ be a subset of $S_0$. 
\begin{align}\label{useful1}
\sum_{j>0}\sum_{I: |I|= 2^{-j_l-j}}|I|
d_j(5I, T_j) d_j(I, P_j) \leq C |I_T|
\end{align}
and
\begin{align}\label{useful2}
 \sum_{j>0}\sum_{I: |I|= 2^{-j_l-j}}|I|
d_j(5I, P_j)d_j(I, T_j)
\leq C |I_T|.
\end{align}
\end{lemma}
\begin{proof} We only prove (\ref{useful1}) and (\ref{useful2}) can be treated exactly the same.
Recall that $T_j=\{n: (j, n)\in T\}$. We identify it with the union of all intervals $I_{n, l, j}$.
This union can be written as a union of finitely many connected components (intervals). 
We use $\alpha_j$ to denote the number of such connected components.   Then we have 
\begin{align}\label{cool}
\sum_{j>0} \alpha_j 2^{-j_l-j} \leq C |I_T|.
\end{align}
The proof can be found in Lemma 4.8 in \cite{mu2002} or  Lemma 4.17 of \cite{li2008U}. Therefore, it is sufficient for us to prove
\begin{align}\label{target00}
\sum_{I: |I|= 2^{-j_l-j}}d_j(5I, T_j)d_j(I, P_j)
\leq C \alpha_j.
\end{align}
Consider the following 3 cases:\\
$\bullet$ {Case 1}: $I\subset T_j$,
\\
$\bullet$ {Case 2}: $5I\cap T_j \neq \emptyset$ but $ I\cap T_j = \emptyset$,
 \\
$\bullet$ {Case 3}: $2^{k-1}|I| \leq {\rm dist}(5I, T_j) \leq 2^k|I|$, $k=1,2,3,\cdots $\\
We show for each case above, (\ref{target00}) holds. For Case 1, for each component $T_j(\ell)$ of $T_j$, $1\leq \ell \leq \alpha_j$. We have
$$
\sum_{I: I\subset T_j(\ell)} d_j(I,P_j) \leq C,
$$
since $T_j\cap P_j = \emptyset$. Henceforth
$$
\sum_{\ell=1}^{\alpha_j}\sum_{I:I\subset T_j(\ell)} d_j(I,P_j)\leq C\alpha_j.
$$
For the contribution of Case 2, notice that the cardinality of $I$ with  ${\rm dist}(I,T_j) \leq 3|I|$ and $I\cap T_j= \emptyset$ is bounded by $8\alpha_j$, which yields the desired estimate for the case.
\\
Finally, we turn to the contribution of the last case. Indeed
$$
\#\{I :  2^{k-1} |I|\leq {\rm dist} (5I,T_j)\leq 2^k|I| 	\} \leq C2^k \alpha_j.
$$
This implies
$$
\sum_{I :  2^{k-1} |I|\leq {\rm dist} (5I,T_j)\leq 2^k|I| }d_j(5I,T_j) \leq C2^k\alpha_j(1+2^{k+m-1})^{-K}.
$$
Therefore
\begin{align}
\sum_{k=1}^{\infty}\sum_{I :  2^{k-1} |I|\leq {\rm dist} (5I,T_j)\leq 2^k|I| }d_j(5I,T_j)
\leq \sum_{k=1}^{\infty}C2^k\alpha_j(1+2^{k+m-1})^{-K}
\leq C\alpha_j.
\end{align}
as desired.
\end{proof}

We now turn to the proof of Lemma {\ref{remove00}}. We first dominate 
\begin{align*}
|\Lambda_{P\cup T}(f,g)-\Lambda_{P}(f,g)-\Lambda_{T}(f,g)| \leq {\rm I}+{\rm II}
\end{align*}
where 
\begin{align*}
{\rm I} = \sum_{j>0} \iint \sum_{n\in P_j}\Big|f_{n,l,m,j}(x-\textbf{tr}_j(t)) \rho(t)\Big| 
\Big|\sum_{n\in T_j}g_{n,l,m,j}(x)\Big|dxdt
\end{align*}
and 
\begin{align*}
{\rm II} = \sum_{j>0}   \iint \sum_{n\in T_j}\Big|f_{n,l,m,j}(x-\textbf{tr}_j(t)) \rho(t)\Big|
\Big|\sum_{n\in P_j}g_{n,l,m,j}(x)\Big|dxdt .
\end{align*}
We estimate $\rm I$ by using H\"oder's Inequality
\begin{align*}
{\rm I} \leq \sum_{j>0} \sum_{|I|= 2^{-j_l-j}} 
\Big\|\int \sum_{n\in P_j}|f_{n,l,m,j}(\cdot -\textbf{tr}_j(t)) \rho(t)|dt     \Big\|_{L^{p_2'}(I)}
 \Big\| \sum_{n\in T_j}g_{n,l,m,j}  \Big\|_{L^{p_2}(I)}.
\end{align*}
The first factor in the sum satisfies
\begin{align*}
  \numberthis \label{1-fac}& \Big\| \int\sum_{n\in P_j}|f_{n,l,m,j}(\cdot-\textbf{tr}_j(t)) \rho(t)|dt     \Big\|_{L^{p_2'}(I)}
\\
\leq &
\sum_{n\in P_j}\Big\|\int |f_{n,l,m,j}(\cdot-\textbf{tr}_j(t)) \rho(t)|dt     \Big\|_{L^{p_2'}(I)}
\\
\leq &
\sum_{n\in P_j}d_j(5I, I_{n,l,j})\cdot \tilde{\rm I}\,,
\end{align*}
where $\tilde{\rm I}$ is defined as 
\begin{equation}
\tilde {\rm I}= \Big \|\int \Big|\Id^{**}_{n,l,j}\psi_{j_l+j+m}f\ast \Phi_{j_l+j+m}(\cdot-\textbf{tr}_j(t)) \rho(t)\Big|dt    
 \Big\|_{L^{p_2'}(\mathbb R)}\,.
\end{equation}

Notice that 
\begin{align*}
\tilde{\rm I} \leq C\Big\|\Id^{**}_{n,l,j}\psi_{j_l+j+m}f\ast \Phi_{j_l+j+m}  
 \Big\|_{L^{p_2'}(\mathbb R)} \leq |I|^{1/p_2'}{\rm \textbf{size}}_1^*(P\cup T).
\end{align*}
Here, the first inequality follows from
\begin{align*}
\int \Big|\Id^{**}_{n,l,j}\psi_{j_l+j+m}f\ast \Phi_{j_l+j+m}(x-\textbf{tr}_j(t)) \rho(t)\Big|dt \leq CM\big(\Id^{**}_{n,l,j}\psi_{j_l+j+m}f\ast \Phi_{j_l+j+m}\big)(x)
\end{align*}
and the $L^p$ boundednes of Hardy-Littlewood maximal function, the second inequality is a consequence of the definition of ${\rm \textbf{size}}^*_1$.\\
Observe that  
\begin{align*}
\sum_{n\in P_j}d_j(5I, I_{n,l,j}) \leq Cd_j(5I, P_j)\,. 
\end{align*}
Hence, $(\ref{1-fac})$ is dominated by 
\begin{equation}\label{1-fac-1}
 (\ref{1-fac}) \leq C d_j(5I, P_j) |I|^{1/p_2'}{\rm \textbf{size}}_1^*(P\cup T).
\end{equation}
On the other hand, by the definition of ${\rm \textbf{ size}}_2$, we have 
\begin{equation}\label{2-fac}
 \Big\| \sum_{n\in T_j}g_{n,l,m,j}  \Big\|_{L^{p_2}(I)}\leq C|I|^{1/p_2}{ \rm\textbf{ size}}_2(P\cup T)d_j(I, T_j)\,. 
\end{equation}
From (\ref{1-fac-1}) and (\ref{2-fac}), we then estimate ${\rm I}$ by
\begin{equation}\label{111-est-111}
 {\rm \textbf{size}}_1^*(P\cup T){\rm \textbf{size}}_2(P\cup T) \sum_{j>0}\sum_{|I|= 2^{-j_l-j}}  
|I|d_j(5I, P_j)d_j(I, T_j)\,. 
\end{equation}

Similarly, we can majorize ${\rm II}$ by 
\begin{equation}\label{111-est-222}
 {\rm \textbf{size}}_1^*(P\cup T){\rm \textbf{size}}_2(P\cup T) \sum_{j>0}\sum_{|I|= 2^{-j_l-j}}  
|I|d_j(I, P_j)d_j(5I, T_j)\,. 
\end{equation}

Applying Lemma \ref{claim1} to (\ref{111-est-111}) and (\ref{111-est-222}), we obtain the desired result.

%
%
%
%
%
%
%
%
%
%
%
%
%
\appendix
\section{Stationary phase with perturbation}
\setcounter{equation}0

In this section, we provide a detailed discussion of the stationary phase with perturbation and the main goal is to prove Lemma \ref{perturbation1} and
 Lemma \ref{derivative00} in Section \ref{section5}. Through out this section, we assume $\frac{1}{2}< |t| < 2 $. We need some preparation before we 
dive into the proof. Let
\begin{align}\label{phase2}
\int e^{-2 \pi i 2^m(t\xi + (t^l+ Q_l(t))\eta)}\rho(t)dt \sim 2^{-m/2}\mathcal{N}_m (\xi,\eta)= 2^{-m/2}e^{2^m i \phi_l(\xi,\eta)} 
\end{align}
and
\begin{align}\label{phase*}
\int e^{-2 \pi i 2^m(t\xi + t^l\eta)}\rho(t)dt \sim 2^{-m/2} {\mathcal{N}^*_m} (\xi,\eta)= 2^{-m/2}e^{2^m i \phi^*_l(\xi,\eta)},  
\end{align}
where (\ref{phase*}) is a standard one, (\ref{phase2}) is the perturbed version. $\phi_l(\xi,\eta)$ is defined in (\ref{phase0}) and $\phi^*_l(\xi,\eta)$ is defined similarly. Indeed, $\phi^*_l(\xi,\eta) = c_l\frac{\xi^{l/(l-1)}}{\eta^{1/(l-1)}}$.

For convenience, we introduce the following terminology. 
\begin{definition}\label{pair1}
Let $\mathcal J=(1/2,2)$ or $(-2,-1/2)$, $K$, $N$ be two large positive integers and $\|\cdot\|_{D_K(\mathcal J)}$ defined as in Definition \ref{pair0}.  
We say two smooth functions $F_0$ and $F_1$ are $(K,N)$-pair on $\mathcal J$ if the following conditions are true:
\begin{flalign*}
&(i)\q\,\,\, \inf_{t\in \mathcal J}|D(F_0)(t)|\gtrsim 1 \, \mbox{and}\q\inf_{t\in \mathcal J}|D(F_1)(t)| \gtrsim 1,&
\\
&(ii)\q\,\,  \|F_0\|_{D_K(\mathcal J)} \lesssim 1  \q\,\,\, \mbox{and}\q
\|F_1\|_{D_K(\mathcal J)} \lesssim 1,
\\
&(iii)\q \|F_0-F_1\|_{D_K(\mathcal J)} \leq 2^{-N}
\end{flalign*}
\end{definition} 
In the above definition, $X\lesssim Y$ means there is a constant depends only on $d$ such that $X\leq CY$. The following lemma is the main ingredient of this appendix. It states that if two functions are "closed" up to $K$-th derivative, then their inverses are "closed" up to $(K-1)$-th derivative. 
%
%
\begin{lemma}\label{lemmakey}
Let $F_0$ and $F_1$ be $(K,N)$-pair in $\mathcal J$ and suppose that $N>>2^{(100(K+1)d)!}$. Let $A\subset \mathbb R$ such that for every $a\in A$, 
there exists a unique $F^{-1}_0(a)\in \mathcal J$ and a unique $F^{-1}_1(a)\in \mathcal J$. If $N$ is sufficiently large (depends on $K$), then
$$\|F_0^{-1}-F_1^{-1}\|_{D_{K-1}(\mathcal J)}\leq 2^{-N/3}.$$
\end{lemma}

The following lemma provides a description of the $n$-th derivative of the inverse function, the proof can be found in \cite{john} or obtained by induction. We omit it. 
\begin{lemma}\label{highorder}.

Let $y = F(x)$ be a smooth function, then for any $n\geq 1$ the $n$-th derivative of $F^{-1}$ can be represented as
\begin{align}
\frac{d^nx}{dy^n}  =
\sum\limits_{k=0}^{n-1}(f'(x))^{-n-k}R_{n,k}(f''(x),\dots, f^{(n-k)}(x))
\end{align}
where $R_{n,k}(f''(x),\dots, f^{(n-k+1)}(x))$ is a polynomial whose coefficients depend only on $n$ and $k$. 

\end{lemma}


%
We now prove Lemma \ref{lemmakey}.
\begin{proof}

Let $t_0=t_0(a)=F_0^{-1}(a)$ and $t_1=t_1(a)=F_1^{-1}(a)$ be the roots of $F_0(t) =a$ and $F_1(t) =a$ respectively. Then it is not difficult to see, 
\begin{align*}
|t_0 - t_1| =|F_0^{-1}(a)-F_1^{-1}(a)|\leq 2^{-N/2 }
\end{align*}
for all $a$ satisfying the conditions of the lemma. 
By Lemma \ref{highorder} we obtain that for $1\leq n \leq K-1$ 
$$
D^n(F_0^{-1})(a)=\sum\limits_{k=0}^{n-1}\big(F_0'(t_0)\big)^{-n-k}R_{n,k}\big(F_0''(t_0),\dots, F_0^{(n-k)}(t_0)\big)
$$
and
$$
D^n(F_1^{-1})(a)=\sum\limits_{k=0}^{n-1}\big(F_1'(t_1)\big)^{-n-k}R_{n,k}\big(F_1''(t_1),\dots, F_1^{(n-k)}(t_1)\big).
$$
Thus 
\begin{align}\label{difference}
D^n(F_0^{-1})(a)-D^n(F_1)^{-1}(a) 
=\Big(F'_0(t_0)F'_1(t_1)\Big)^{-2n}\sum\limits_{k=1}^n \Big(F_0^{(k)}(t_0)-F_1^{(k)}(t_1)\Big)\mathcal R_{n,k}  
\end{align}
where $\mathcal R_{n,k}$ is a polynomial of $F'_0(t_0), F''_0(t_0),\dots ,F^{(n)}_0(t_0)$ and $ F'_1(t_1), F''_1(t_1),\dots ,F^{(n)}_1(t_1)$. By Definition \ref{pair1} (ii), the absolute values of  $F'_0(t_0), F''_0(t_0),\dots ,F^{(n)}_0(t_0)$ and $ F'_1(t_1), F''_1(t_1),$ $\dots ,F^{(n)}_1(t_1)$ are bounded above by $\sim 1$ and Thus $|\mathcal{R}_{n,k}|\leq C_K$.  
Observe that
$$
F_0^{(k)}(t_0)-F_1^{(k)}(t_1)= (F_0^{(k)}(t_0)-F_1^{(k)}(t_0))+(F_1^{(k)}(t_0)-F_1^{(k)}(t_1)).
$$
Hence from definition \ref{pair1} (iii), for $1\leq k\leq K$ we have $|(F_0^{(k)}(t_0)-F_1^{(k)}(t_0))| \leq 2^{-N}$. The mean value theorem 
and Definition \ref{pair1} (iii) yield, for $1\leq k\leq K-1$, 
\begin{align}\label{difference1}
|(F_1^{(k)}(t_0)-F_1^{(k)}(t_1))|  
=|F_1^{(k+1)}(t')||t_0-t_1| \leq C_K2^{-N/2},
\end{align}
where $t'$ is a number between $t_0$ and $t_1$. 
By definition \ref{pair1} (i) we have 
\begin{align}\label{difference2}
|(F'_0(t_0)F'_1(t_1))^{-2n}| \leq  C_K
\end{align}
Choosing $N$ large enough such that it overrides all the above constants, then combining (\ref{difference}), (\ref{difference1}) and (\ref{difference2}) we obtain lemma \ref{lemmakey}.
\end{proof}
\begin{remark}
We'll choose $K$ depends on $d$, say $K=100d^2$ and thus $N$ depends only on $d$. For the rest of this section, 
 we use $\|\cdot\|_{D_{K}}$ to mean $\|\cdot\|_{D_{K}(\mathcal J)}$.
\end{remark}

\subsection{Proof Lemma \ref{perturbation1}}

Let $t_0=t_0(\frac{\xi}{\eta})$ and $t_1=t_1(\frac{\xi}{\eta})$ be the roots of $F_0(t) = D(\xi t +(t^l+Q_l(t))\eta)$ and $F_1(t) =D( t\xi+t^l\eta)$ respectively, equivalently $t_0$ and $t_1$ are the inverses of $F_0$ and $F_1$. A direct calculation shows that $F_0(t)$ and $F_1(t)$ are $(K, N/2)$-pair, for $K=100d^2$. Notice

\begin{align*}
\phi_l(\xi,\eta) &= C_l \left( \frac{\xi}{\eta}t_0+t_0^l +Q_l(t_0)\right)\eta
\\
&=  C_l \left( \frac{\xi}{\eta}t_1+t_1^l\right)\eta +C_l\left( \frac{\xi}{\eta}(t_0-t_1)+(t_0^l-t_1^l)+Q_l(t_0)\right)\eta
\\
&=\phi_l^*(\xi,\eta)+C_l\left( \frac{\xi}{\eta}(t_0-t_1)+(t_0^l-t_1^l)+Q_l(t_0)\right)\eta
\\
:&=\phi_l^*(\xi,\eta)+\varepsilon(\frac{\xi}{\eta})\eta. 
\end{align*} 
Then 
\begin{align*}
Q_\tau(u,v) =\left( \phi_l^*(u,v)-\phi_l^*(u-\tau,v+b_2\tau)\right) +\left(\varepsilon(\frac{u}{v})v-\varepsilon(\frac{u-\tau}{v+b_2\tau})(v+b_2\tau)\right).
\end{align*}
One can show 
$$
\Big|\partial_u\partial_v\Big( \phi_l^*(u,v)-\phi_l^*(u-\tau,v+b_2\tau)\Big)\Big|\geq c_l|\tau|,
$$
 see \cite{li2008} for the details. By Lemma \ref{lemmakey} for $0\leq n\leq K-1$, we have $|D^n(t_0-t_1)| \leq 2^{-N/6}$. Also notice that $|D^nQ_l(t_0)| \leq 2^{-N/2}$ and $|u|,|v|, |t_0|, |t_1|, |u-\tau|,|v+b_2\tau |\sim 1$. These bounds are enough to guarantee that when $N$ is large enough,
$$
\Big|\partial_u\partial_v \Big(\varepsilon(\frac{u}{v})v-\varepsilon(\frac{u-\tau}{v+b_2\tau})(v+b_2\tau)\Big)\Big|
\leq 2^{-N/10}|\tau| \leq \frac{c_l}{2}|\tau|.
$$ 
Therefore
$$
|\partial_u\partial_v Q_\tau(u,v)| \geq c_l|\tau| -\frac{c_l}{2} |\tau| =\frac{c_l}{2}|\tau|. 
$$

\subsection{Proof of Lemma \ref{derivative00}}

Let $\tilde\theta_l(z) = {z}^\frac{l}{l-1}$. A direct calculation shows that $D\tilde\theta_l$ and $D\theta_l$ is an $(K,N)$-pair,
thus Lemma \ref{lemmakey} yields
$$
\Big\|(D\tilde\theta_l)^{-1} -(D\theta_l)^{-1}\Big\|_{D_{K-1}} \leq 2^{-N/3}.
$$
Let $\zeta(z) =(D\theta_l)^{-1}(z)$ and $\tilde\zeta(z) = (D\tilde\theta_l)^{-1}$, then $\|\zeta - \tilde\zeta\|_{D_{K-1}} \leq 2^{-N/3}$. Let
$\beta(z) = \theta_l(\zeta(z))+z\zeta(z)$ and
$\tilde\beta(z) = \tilde\theta_l(\tilde\zeta(z))+z\tilde\zeta(z)$. If $|z|\sim 1$, then 
$$
\|\beta - \tilde\beta\|_{D_{K-2}} \leq 2^{-N/5}. 
$$
Let $\tilde\kappa_l(s) = s^{\frac{1}{l}}$. Similarly for $|s| \sim 1$, one has $\|\kappa_l -\tilde\kappa_l\|_{D_{K-1}} \leq 2^{-N/2}$. Define
$$
 \tilde{\mathcal O_\tau}(u,v) = \tilde\beta\Big(c(u-\tilde\kappa_l(s)+b')\Big) - \tilde\beta\Big(c(u+2^{-(l-1)j}\tau-\tilde\kappa_l(s+\tau)+b')\Big).
$$
Since, $\tilde\beta$, $\tilde\kappa$ are standard functions (without perturbation perturbation). The following result comes from direct computation: 
$$
|\partial^{l-1}_u\partial_v \tilde{\mathcal O_\tau}(u,v)| \geq c_l|\tau|.
$$
Moreover, if $l=2$ we have
$$
|\partial_u\partial^2_v \tilde{\mathcal O_\tau}(u,v)| \geq c|\tau|.
$$
Again, details can be found in \cite{li2008}. It remains to estimate the error terms:
$$
|\partial^{l-1}_u\partial_v (\tilde{\mathcal O_\tau}(u,v)-\mathcal O_\tau(u,v))|\leq 2^{-N/100}|\tau|
$$
and
$$
|\partial_u\partial^2_v (\tilde{\mathcal O_\tau}(u,v)-\mathcal O_\tau(u,v))|\leq 2^{-N/100}|\tau|, \q{\rm if }\q l =2.
$$
We handle the first one and the second one is similar. 
Let 
$$
\varepsilon(u,v) =  \tilde\beta\Big(c(u-\tilde\kappa_l(v)+b')\Big) - \beta\Big(c(u-\kappa_l(v)+b')\Big),
$$
then
\begin{align*}
\tilde{\mathcal O_\tau}(u,v)-\mathcal O_\tau(u,v)
=
\varepsilon(u,v)-\varepsilon(u+2^{-(l-1)j}\tau,v+\tau)
\end{align*}
and $\partial^{l-1}_u\partial_v
\Big(\tilde{\mathcal O_\tau}(u,v)-\mathcal O_\tau(u,v)\Big)
$ equals
\begin{align*} 
&
\Big(\partial^{l-1}_u\partial_v\varepsilon(u,v)-\partial^{l-1}_u\partial_v\varepsilon(u+2^{-(l-1)j}\tau,v)\Big)+
\\
&\q\q\q \Big(\partial^{l-1}_u\partial_v\varepsilon(u+2^{-(l-1)j}\tau,v)-\partial^{l-1}_u\partial_v\varepsilon(u+2^{-(l-1)j}\tau,v+\tau)\Big).
\end{align*}
By the mean value theorem, the absolute value of the above expression is bounded by
\begin{align*}
|(\partial^{l}_u\partial_v\varepsilon(u+\lambda_12^{-(l-1)j}\tau,v)||2^{-(l-1)j}\tau|
+
|(\partial^{l-1}_u\partial^2_v\varepsilon(u+2^{-(l-1)j}\tau,v+\lambda_2\tau)||\tau|.
\end{align*}
Hence, it remains to prove 
\begin{align*}
|\partial^{l}_u\partial_v\varepsilon(u,v)|\leq 2^{-N/50}\q
 \mbox{and }\q
|\partial^{l-1}_u\partial^2_v\varepsilon(u,v)|\leq 2^{-N/50}.
\end{align*}
Notice
\begin{align*}
\partial^{l}_u\partial_v\varepsilon(u,v) 
=
\tilde\beta^{(l+1)}\Big(c(u-\tilde\kappa_l(v)+b')\Big)\tilde\kappa'(v) - \beta^{(l+1)}\Big(c(u-\kappa_l(v)+b')\Big)\kappa'(v)
\\
\q\q=
\left(\tilde\beta^{(l+1)}\Big(c(u-\tilde\kappa_l(v)+b')\Big)\tilde\kappa'(v) -\tilde \beta^{(l+1)}\Big(c(u-\kappa_l(v)+b')\Big)\tilde\kappa'(v)\right)
+
\\
\left(\tilde\beta^{(l+1)}\Big(c(u-\kappa_l(v)+b')\Big)\tilde\kappa'(v) - \beta^{(l+1)}\Big(c(u-\kappa_l(v)+b')\Big)\kappa'(v)\right)
\end{align*}
Since $\| \kappa-\tilde\kappa\|_{D_{K-1}} \leq 2^{-N/2}$, a use of mean value theorem yields
 that the absolute value of the above expression is majored by $2^{-N/50}$. Similarly $|\partial^{l-1}_u\partial^2_v\varepsilon(u,v)|\leq 2^{-N/50}$.

\bibliographystyle{plain}
\bibliography{bib}

\end{document}